\documentclass[11pt]{article}

\usepackage[english]{babel}
\usepackage{amsmath,amsfonts,amssymb,amsthm}
\usepackage{latexsym}
\usepackage{makeidx}
\usepackage{amscd}

\usepackage{vmargin}
\setmarginsrb{20mm}{10mm}{20mm}{20mm}%
            {8mm}{8mm}{8mm}{8mm}

\numberwithin{equation}{section}

\numberwithin{table}{section}

\usepackage{caption}

\usepackage[all]{xy}

\newtheorem{em-deff}{Definition}[section]

\newtheorem{lemma}[em-deff]{Lemma}
\newtheorem{theorem}[em-deff]{Theorem}
\newtheorem{corollary}[em-deff]{Corollary}

\newtheorem{proposition}[em-deff]{Proposition}
\newtheorem{em-fact}[em-deff]{Fact}
\newtheorem{em-example}[em-deff]{Example}

\newtheorem{problem}[em-deff]{Problem}

\newtheorem{em-remark}[em-deff]{Remark}

\newenvironment{example}{\begin{em-example} \em }{ \end{em-example}}

\newenvironment{deff}{\begin{em-deff} \em }{ \end{em-deff}}
\newenvironment{fact}{\begin{em-fact} \em }{ \end{em-fact}}

\newcommand{\N}{\mathbb N}
\newcommand{\Z}{\mathbb Z}
\newcommand{\Q}{\mathbb Q}

\newcommand{\F}{\mathcal F}
\newcommand{\R}{\mathbb R}

\newcommand{\J}{\mathbb J}

\def\f{\phi}
\def\ent{\mathrm{ent}}
\def\supp{\mathrm{supp}}
\def\Per{\mathrm{Per}}
\def\End{\mathrm{End}}

\def\cont{\mathfrak c}
\def\Sns{\mathfrak S_{ns}}
\def\S0{\mathfrak S_0}
\global\def\sc#1{\mathrm{sc}\,{#1}}

\title{String numbers of abelian groups}

\author{Anna Giordano Bruno
\\{\footnotesize {\tt  anna.giordanobruno@math.unipd.it}} 
\\{\footnotesize Dipartimento di Matematica Pura e Applicata,}
\\{\footnotesize Universit\`a di Padova,}
\\{\footnotesize Via Trieste, 63 - 35121 Padova}
 \and Simone Virili
 \\{\footnotesize {\tt Simone@mat.uab.cat}}
\\{\footnotesize Departament de Matematiques,}
\\{\footnotesize Universitat Autonoma de Barcelona, }
\\{\footnotesize Edifici C  - 08193 Bellaterra (Barcelona)}
 }

\date{}

\begin{document}

\maketitle


\abstract{The string number of self-maps arose in the context of algebraic entropy and it can be viewed as a kind of combinatorial entropy function. Later on its values for endomorphisms of abelian groups were calculated in full generality. We study its global version for abelian groups, providing several examples involving also Hopfian abelian groups. Moreover, we characterize the class of all abelian groups with string number zero in many cases and discuss its stability properties.}

\section{Introduction}

The strings and the string number of a self-map of a set were introduced in \cite{AADGH}, in order to compute the algebraic entropy of particular group endomorphisms (see Section \ref{ent-sec} for the definition of algebraic entropy). Then, motivated by open questions in \cite{AADGH} and results in \cite{GB}, in \cite{DGV} the string number of endomorphisms of abelian groups was studied and described in full generality. In this paper we study the ``global version'' of the string number for abelian groups.

\medskip
We now recall the definition of strings (from \cite{AADGH}) and of particular kinds of strings (from \cite{DGV}) recast in the context of abelian groups:

\begin{deff}
Let $G$ be an abelian group and $\phi:G\to G$ an endomorphism. A sequence $S=\{x_n\}_{n\in\N}\subseteq G$ is
\begin{itemize}
\item[(a)] a \emph{pseudostring} of $\phi$ if $\phi(x_n)=x_{n-1}$ for every $n\in\N_+$;
\item[(b)] a \emph{string} of $\phi$ if $S$ is a pseudostring and its elements are pairwise distinct;
\item[(c)] a \emph{singular string} of $\phi$ if $S$ is a string such that  $\phi^j(x_0)=\phi^k(x_0)$ for some $j\neq k$ in $\N$;
\item[(d)] a \emph{null string} of $\phi$ if  $S$ is a string such that $x_0\neq0$ and $\phi^k(x_0)=0$ for some $k\in\N_+$.
\end{itemize}
\end{deff}

In \cite{AADGH} a function was defined to measure the number of strings, and analogous functions were introduced in \cite{DGV} for non-singular strings and null strings. We now recall the definition of these functions, that in general we call string numbers.

\begin{deff}
Let $G$ be an abelian group and $\phi:G\to G$ an endomorphism. Then:
\begin{itemize}
\item[(b$'$)] $s(\phi)=\sup\left\{|\mathcal F|: \mathcal F\ \text{is a family of pairwise disjoint strings of $\phi$}\right\}$ is the \emph{string number} of $\phi$;
\item[(c$'$)] $ns(\phi)=\text{sup}\left\{|\mathcal{F}|:\text{$\mathcal{F}$ is a family of pairwise disjoint non-singular strings of $\phi$}\right\}$ is the \emph{non-sin\-gu\-lar string number} of $\phi$;
\item[(d$'$)] $s_0(\phi)=\sup\left\{|\mathcal F|: \mathcal F\ \text{is a family of pairwise disjoint null strings of $\phi$}\right\}$ is the \emph{null string number} of $\phi$.
\end{itemize}
\end{deff}

When the suprema in this definition are infinite, we set them equal to the symbol $\infty$. In the main theorems of \cite{DGV} it is proved that actually 
\begin{center}
\emph{the string numbers of endomorphisms of abelian groups have values in $\{0,\infty\}$}, 
\end{center}
and so we distinguish their values only between zero and infinity.

Note that the same surprising dichotomy for the values was proved in \cite{DGS} for the adjoint algebraic entropy (see Section \ref{ent-sec} for the definition of adjoint algebraic entropy).

\medskip
In analogy to what is done for the algebraic entropy in \cite{DGSZ} and for the adjoint algebraic entropy in \cite{DGS} and in \cite{GG} (see Section \ref{ent-sec}), it is possible to introduce the following ``global notions'' of string numbers of abelian groups, as noted in \cite[Definition 1.6]{DGV}.

\begin{deff}
Let $G$ be an abelian group.
\begin{itemize}
\item[(b$''$)] The \emph{string number} of $G$ is $s(G)=\sup\{s(\f):\f\in\End(G)\}$.
\item[(c$''$)] The \emph{non-singular string number} of $G$ is $ns(G)=\sup\{ns(\f):\f\in\End(G)\}$.
\item[(d$''$)] The \emph{null string number} of $G$ is $s_0(G)=\sup\{s_0(\f):\f\in\End(G)\}$.
\end{itemize}
\end{deff}

These functions take values in $\{0,\infty\}$. This follows immediately from the dichotomy of the values of the string numbers of endomorphisms of abelian groups, so again we distinguish the values of the string numbers of abelian groups only between zero and infinity. 

\medskip
We are interested in studying the abelian groups with one of the string numbers zero, and so we introduce the following classes: 
$$
\mathfrak S=\{G\ \text{abelian group}: s(G)=0\}, \ \ \  \Sns=\{G\ \text{abelian group}: ns(G)=0\}\ \text{and}$$ $$
\S0=\{G\ \text{abelian group}: s_0(G)=0\}.
$$
It is clear that $\mathfrak S\subseteq \Sns\cap \S0$, and we see in  Section \ref{examples} that actually $\mathfrak S= \Sns\cap \S0$. In this paper we study the following problem given in \cite[Problem 1.7]{DGV}.

\begin{problem}\label{pb}
Characterize the classes $\mathfrak S$, $\Sns$ and $\S0$.
\end{problem}

The counterpart of this problem for the algebraic entropy was studied in \cite{DGSZ}, and for the adjoint algebraic entropy in \cite{GG} and \cite{SZ}.

\medskip
In Section \ref{examples} we give many examples in which we calculate the values of the string numbers; we collect these examples in Table \ref{table1} below. Moreover, we find first basic properties of the string numbers. In particular, Proposition \ref{fg} and its Corollary \ref{free} solve completely Problem \ref{pb} for finitely generated abelian groups and for free abelian groups respectively. 

\medskip
Section \ref{s=0-tor} is dedicated to Problem \ref{pb} in the case of torsion abelian groups.
Theorem \ref{ThB} characterizes completely $\mathfrak S\cap \mathfrak T$ and $\mathfrak S_0\cap \mathfrak T$, where $\mathfrak T$ denotes the class of all torsion abelian groups. Indeed, $\mathfrak S\cap \mathfrak T=\mathfrak S_0\cap \mathfrak T$, and these two classes coincide with the class of all torsion abelian groups with all finite $p$-components.

With respect to Problem \ref{pb} for the non-singular string number, Corollary \ref{B-2} shows that for a torsion abelian group $G$ the condition 
\begin{equation}\label{cond}
\text{all the $p$-ranks of $G$ are finite}
\end{equation}
is a sufficient condition for $G$ to belong to $\Sns$. But this is not a necessary condition, as Example \ref{hopf-s0=infty} shows. So the following problem remains open.

\begin{problem}\label{Ques1}
Characterize the class $\Sns\cap \mathfrak T$.
\end{problem}

Nevertheless, we give some reductions and partial results, restricting this problem.
For example, in Theorem \ref{||<c} we see that, if $G$ belongs to $\Sns\cap \mathfrak T$, then $G$ has size less or equal than the cardinality of continuum $\cont$. 
Moreover, in the last part of Section \ref{s=0-tor}, we underline the difficulty of Problem \ref{Ques1}; indeed, we see that this problem is related to the classification of abelian groups with zero algebraic entropy, that was shown to be difficult in \cite{DGSZ}. 

\medskip
Section \ref{tf-sec} is dedicated to the torsion-free case of Problem \ref{pb}. In particular, Proposition \ref{tf-pomega} shows that $p^\omega G=0$ for every prime $p$ is a necessary condition for a torsion-free abelian group $G$ to belong to $\mathfrak S$. Moreover, this condition turns out to be also sufficient in case $G$ is endorigid (see Theorem \ref{endorigid}). Following \cite{EM}, we say that an abelian group $G$ is \emph{endorigid} if $\End(G)=\Z$; note that, since the ring $\Z$ does not contain idempotents, endorigid implies indecomposable.

On the other hand, Theorem \ref{tf--} shows that the problem of finding a complete classification of torsion-free abelian groups in $\mathfrak S$ and in $\mathfrak S_{ns}$ is a deep one, since for every infinite cardinal $\lambda$ it is possible to find both an indecomposable torsion-free abelian group $G$ of torsion-free rank $\lambda$ with $s(G)=ns(G)=0$, and an indecomposable torsion-free abelian group $H$ of torsion-free rank $\lambda$ with $s(H)=ns(H)=\infty$.

\medskip
In Section \ref{stab-sec} we consider the closure properties of the classes $\mathfrak S$, $\Sns$ and $\S0$. In particular we discuss in detail the closure of these classes under taking subgroups, quotients, direct summands,  extensions, direct sums and direct products. We show with various examples that the unique closure property possessed by the classes $\mathfrak S$, $\Sns$ and $\S0$, among the ones listed above, is that under taking direct summands. However, we prove that with additional hypotheses we can obtain closure properties in some particular cases. 

\medskip
A group is \emph{Hopfian} if each of its surjective endomorphisms is an automorphism. Equivalently, an Hopfian group is not isomorphic to any of its proper quotients. Easy examples of Hopfian groups are finite groups, finitely generated groups, torsion-free abelian groups of finite torsion-free rank and endorigid abelian groups. We discuss the connections between the string numbers and the Hopfian property in Section \ref{SHopf}, giving various examples. In particular, we show that $\S0$ (and so also $\mathfrak S$) is contained in the class of all Hopfian abelian groups, while $\Sns$ is not contained in that class. Moreover, Example \ref{hopf-s0=infty} exhibits a Hopfian abelian group not belonging to $\S0$, and this is a negative answer to \cite[Question 3.13]{DGV}.

\medskip
As we said above, the classes $\mathfrak S$, $\Sns$ and $\mathfrak S_0$ are not closed under taking subgroups. So we introduce three monotone functions generated by the three usual string numbers, namely, the \emph{hereditary string numbers}. Indeed, if $i(-)$ is any function defined on abelian groups with values in $\R_{\geq0}\cup\{\infty\}$ (e.g., $i(-)$ is one among $s(-)$, $ns(-)$, $s_0(-)$), then it is possible to define its ``hereditary modification'' setting, for every abelian group $G$,
$$\widetilde i(G)=\sup\{i(H):H\leq G\}.$$ 
Clearly, $\widetilde i(G)\geq i(G)$, and $\widetilde i(-)$ is monotone under taking subgroups. In general, we call any property of abelian groups  preserved under taking subgroups, a hereditary property.

\smallskip
We introduce the following classes also for the hereditary string numbers:
$$
\widetilde{\mathfrak S}=\{G\ \text{abelian group}: \widetilde s(G)=0\},\ \ \ \widetilde{\mathfrak S}_{ns}=\{G\ \text{abelian group}: \widetilde{ns}(G)=0\},\ \text{and}$$
$$\widetilde{\mathfrak S}_0=\{G\ \text{abelian group}: \widetilde{s_0}(G)=0\}.$$
By the definition it follows that $\widetilde{\mathfrak S}\subseteq \mathfrak S$, $\widetilde{\mathfrak S}_{ns}\subseteq \Sns$ and $\widetilde{\mathfrak S}_0\subseteq \S0$. Furthermore, we see that $\widetilde{\mathfrak S}= \widetilde{\mathfrak S}_{ns}\cap \widetilde{\mathfrak S}_0$.

In Section \ref{her-sec} we study the hereditary string numbers and consider the counterpart of Problem \ref{pb} for these smaller classes.
We solve first the torsion and the torsion-free case to come to a complete solution of the general problem in Theorem \ref{main-tilde}.


In particular, in the torsion case we find that $\widetilde{\mathfrak S}\cap\mathfrak T=\mathfrak S\cap \mathfrak T=\widetilde{\mathfrak S}_0\cap\mathfrak T=\mathfrak S_0\cap\mathfrak T$. Moreover, for torsion abelian groups $G$ the condition \eqref{cond},
which is sufficient but not necessary in order to have $ns(G)=0$, is instead equivalent to $\widetilde{ns}(G)=0$; in other words, $\widetilde{\mathfrak S}_{ns}\cap\mathfrak T$ coincides with the class of all torsion abelian groups with all finite $p$-ranks.
In the torsion-free case, denoting by $\mathfrak F$ the class of all torsion-free abelian groups, $\widetilde{\mathfrak S}_0\cap\mathfrak F$ is precisely the class of all torsion-free abelian groups with finite torsion-free rank; moreover, $\widetilde{\mathfrak S}\cap \mathfrak F=\widetilde{\mathfrak S}_{ns}\cap \mathfrak F$, and this class coincides with the class of all torsion-free abelian groups of torsion-free rank $\leq 1$ and with no infinity in their type. 
These partial results are used to prove Theorem \ref{main-tilde}, which describes completely the classes $\widetilde{\mathfrak S}$, $\widetilde{\mathfrak S}_{ns} $ and $\widetilde{\mathfrak S}_0$.


\smallskip
Finally, we find another characterization for the groups in $\widetilde{\mathfrak S}_0$ in terms of the Hopfian property. Indeed, we say that a group is \emph{hereditarily Hopfian} if each of its subgroups is Hopfian. It turns out that an abelian group $G$ is hereditarily Hopfian if and only if $\widetilde{s_0}(G)=0$, and this occurs precisely when all the $p$-components and the torsion-free rank of $G$ are finite (see Corollary \ref{her-H}). This result can be viewed also as a characterization of hereditary Hopfian abelian groups, which are studied also in \cite{GG2}.

%

\subsection{String numbers and algebraic entropies}\label{ent-sec}

We start this section giving the definition of algebraic entropy as suggested in \cite{AKM}, later studied in \cite{W} and recently deeply investigated in \cite{DGSZ}.
Let $G$ be an abelian group and $F$ a finite subgroup of $G$; for an endomorphism $\f:G\to G$ and a positive integer $n$, let $T_n(\f,F)=F+\f(F)+\ldots+\f^{n-1}(F)$ be the \emph{$n$-th $\f$-trajectory} of $F$ and $T(\f,F)=\sum_{n\in\N}\f^n(F)$ the \emph{$\f$-trajectory} of $F$. The \emph{algebraic entropy of $\f$ with respect to $F$} is 
$$H(\f,F)={\lim_{n\to \infty}\frac{\log|T_n(\f,F)|}{n}},$$
and the \emph{algebraic entropy} of $\f$ is 
$$\ent(\f)=\sup\{H(\f,F): F\ \text{is a finite subgroup of } G\}.$$ The \emph{algebraic entropy} of $G$ is $$\ent(G)=\sup \{\ent(\f): \f\in\End(G)\}.$$

\medskip
The adjoint algebraic entropy was defined in \cite{DGS} substituting in the definition of the algebraic entropy the family of all finite subgroups with the family of all finite-index subgroups. We give the precise definition: if $N$ is a finite-index subgroup of an abelian group $G$, $\f:G\to G$ an endomorphism and $n$ a positive integer, the \emph{$n$-th $\f$-cotrajectory} of $N$ is $C_n(\f,N) = \frac{G}{N\cap\f^{-1}(N)\cap\ldots\cap\f^{-n+1}(N)}$.
The \emph{adjoint algebraic entropy of $\f$ with respect to $N$} is 
$$H^\star(\f,N)=\lim_{n\to\infty}\frac{\log|C_n(\f,N)|}{n},$$
and the \emph{adjoint algebraic entropy} of $\f$ is $$\ent^\star(\f)=\sup\{H^\star(\f,N): N\leq G,\ G/N\ \text{finite}\}.$$ The \emph{adjoint algebraic entropy} of $G$ is $$\ent^\star(G)=\sup \{\ent^\star(\f): \f\in\End(G)\}.$$

\medskip
Even if the algebraic entropy of endomorphisms of abelian groups takes all possible values in $\{\log n:n\in\N_+\}\cup\{\infty\}$, in \cite{DGSZ} it is proved that an abelian group $G$ can have algebraic entropy only $0$ or $\infty$. On the other hand, for the adjoint algebraic entropy, this dichotomy for the values holds already at the level of endomorphisms (see \cite{DGS}). 

\medskip
We now give the definition of the three classical Bernoulli shifts. Let $K$ be an abelian group, and for a cardinal (or a set) $\lambda$, we denote by $K^{(\lambda)}$ the direct sum of $\lambda$ copies of $K$.
\begin{itemize}
\item[(a)]  The \emph{two-sided Bernoulli shift} $\overline{\beta}_K$ of the group $K^{(\mathbb Z)}$ is  defined by 
$$ \overline\beta_K((x_n)_{n\in\mathbb Z})=(x_{n-1})_{n\in\mathbb Z}, \mbox{ for every } (x_n)_{n\in\mathbb Z}\in K^{(\mathbb Z)}.$$
\item[(b)] The \emph{right  Bernoulli shift} $\beta_K$ and the \emph{left Bernoulli shift} $_K\beta$ of the group $K^{(\mathbb N)}$ are defined, for every $(x_n)_{n\in\N}\in K^{(\N)}$, respectively by 
$$\beta_K(x_0,x_1,x_2,\ldots)=(0,x_0,x_2,\ldots) \ \mbox{and}\ {}_K\beta(x_0,x_1,x_2,\ldots)=(x_1,x_2,x_3,\ldots).$$ 
\end{itemize}

The left Bernoulli shift and the two-sided Bernoulli shift are relevant examples for both ergodic theory and topological dynamics. The right Bernoulli shift is fundamental for the theory of algebraic entropy (e.g., one of the properties giving uniqueness of the algebraic entropy in $\mathfrak T$ is based on it --- see \cite{DGSZ}). Furthermore, for any abelian group $G$ and any endomorphism $\phi:G\to G$, $\ent(\f)>0$ if and only if for some prime $p$ there exists a subgroup $H$ of $G$, contained in the $p$-socle of $G$, on which $\f$ acts as the right Bernoulli shift (see \cite{DGSZ}). The Bernoulli shifts were fundamental also in proving the dichotomy of the values of the adjoint algebraic entropy (see \cite{DGS} and \cite{GB}).

\smallskip
As noted in \cite{DGV}, the string numbers of the Bernoulli shifts are quite different from the algebraic entropy of the Bernoulli shifts; the values of the string numbers of the Bernoulli shifts are calculated in \cite[Example 3.26]{DGV}. For the values of the algebraic entropy we refer to \cite{DGSZ}, and for the values of the adjoint algebraic entropy to \cite{DGS}. We collect this information in the following table, where $K$ is a non-zero abelian group and we adopt the usual convention that $\log|K|=\infty$ if $K$ is infinite.

\begin{center}
\begin{tabular}{|c|ccc|cc|}
\hline
 & $s(-)$ & $ns(-)$ & $s_0(-)$ & $\ent(-)$ & $\ent^\star(-)$ \\
 \hline
$\beta_K$ & $0$  & $0$ & $0$ & $\log|K|$ & $\infty$ \\
${}_K\beta$ & $\infty$ & $0$ & $\infty$ & $0$ & $\infty$ \\
$\overline\beta_K$ & $\infty$ & $\infty$ & $0$ & $\log|K|$ & $\infty$ \\
\hline
\end{tabular}
\captionof{table}{Bernoulli shifts}\label{bernoulli}
\end{center}

In the context of the string numbers, the Bernoulli shifts are useful since, whenever we have an infinite direct sum of copies of an abelian group $K$, the two sided Bernoulli shift $\overline\beta_K$ is an example of an endomorphism with non-singular strings, whereas the left Bernoulli shift ${}_K\beta$ admits null strings, as shown in Table \ref{bernoulli}. This gives the values of the last line of Table \ref{table1}.

\medskip
In the following table we compare the values of the string numbers, the algebraic entropy and the adjoint algebraic entropy of particular abelian groups. Here $p$ is a prime, $K$ is a non-zero abelian group, and $B_p$ denotes the standard basic subgroup, that is, $B_p=\bigoplus_{n\in\N_+}\Z(p^n)$. The values of the string numbers are calculated in Section \ref{examples}, while for the values of the algebraic entropy and of the adjoint algebraic entropy we refer to \cite{DGSZ} and \cite{DGS} respectively.

\begin{center}
\begin{tabular}{|c|ccc|cc|}
\hline
 & $s(-)$ & $ns(-)$ & $s_0(-)$ & $\ent(-)$ & $\ent^\star(-)$ \\
 \hline
$\Z$ & $0$ & $0$ & $0$ & $0$ & $0$ \\
$\Z^2$ & $\infty$ & $\infty$ & $0$ & $0$ & $0$ \\
$\Q$ & $\infty$ & $\infty$ & $0$ & $0$ & $0$ \\
$\mathbb J_p$ & $\infty$ & $\infty$ & $0$ & $0$ & $0$ \\
$\Z(p^\infty)$ & $\infty$ & $0$ & $\infty$ & $0$  & $0$ \\
$\Q/\Z$ & $\infty$ & $0$ & $\infty$ & $0$ & $0$ \\
$B_p$ & $\infty$ & $\infty$ & $\infty$ & $\infty$ & $\infty$ \\
$K^{(\N)}$ & $\infty$ & $\infty$ & $\infty$ & $\infty$ & $\infty$\\
\hline
\end{tabular}
\captionof{table}{String numbers and entropies}\label{table1}
\end{center}

The next are properties typical of the entropy functions, that hold also for the string numbers. 

\begin{fact}\label{ent*}
\begin{itemize}
\item[(a)] (Conjugation under isomorphism) \cite[Lemma 2.8]{DGV} Let $G$ and $H$ be two abelian groups, and $\xi:G\to G$ an isomorphism. Let $\phi:G\to G$ and $\psi:H\to H$ be conjugated under the action of $\xi$, i.e. $\phi=\xi^{-1} \psi\xi$. Then $s(\phi)=s(\psi)$, $ns(\phi)=ns(\psi)$ and $s_0(\phi)=s_0(\psi)$.  
\item[(b)] (Logarithmic law) \cite[Corollary 4.6]{DGV} Let $G$ be an abelian group and $\f\in\End(G)$. Then $s(\f^k)=k\cdot s(\f)$, $ns(\f^k)=k\cdot ns(\f)$ and $s_0(\f^k)=k\cdot s_0(\f)$ for every $k\in\N_+$.
\item[(c)] (Monotonicity under taking invariant subgroups) \cite[Lemma 2.9]{DGV} Let $G$ be an abelian group, $\phi:G\to G$ an endomorphism and $H$ a $\phi$-invariant subgroup of $G$. Then $s(\phi)\geq s(\phi\restriction_H)$, $ns(\phi)\geq ns(\phi\restriction_H)$ and  $s_0(\phi)\geq s_0(\phi\restriction_H)$.
\item[(d)] (Monotonicity under taking quotients) \cite[Theorem 4.9]{DGV} Let $G$ be an abelian group, $\phi:G\to G$ an endomorphism, $H$ a $\f$-invariant subgroup of $G$ and $\overline\f:G/H\to G/H$ the endomorphism induced by $\f$. Then $s(\f)\geq s(\overline\f)$ and $ns(\f)\geq ns(\overline\f)$. The null string number does not possess monotonicity under taking induced endomorphisms on quotients (see \cite[Example 3.24]{DGV}). 
\item[(e)] (Additivity on direct sums) \cite[Lemma 2.11]{DGV} Let $G$ be an abelian group, $\phi:G\to G$ an endomorphism and $H_1, H_2$ $\phi$-invariant subgroups of $G$ such that $G=H_1\oplus H_2$. Let $\phi_1=\phi\restriction_{H_1}$ and $\phi_2=\phi\restriction_{H_2}$. Then:
\begin{itemize}
\item[(i)] $s(\phi)=0$ if and only if $s(\phi_1)=s(\phi_2)=0$;
\item[(ii)] $ns(\phi)=0$ if and only if $ns(\phi_1)=ns(\phi_2)=0$;
\item[(iii)] $s_0(\phi)=0$ if and only if $s_0(\phi_1)=s_0(\phi_2)=0$.
\end{itemize}
\end{itemize}
\end{fact}


\subsection*{Notation and terminology}

We denote by $\mathbb Z$, $\mathbb N$, $\mathbb N_+$, $\Q$ and $\R$ respectively the set of integers, the set of natural numbers, the set of positive integers, the set of rationals and the set of reals. For $m\in\mathbb N_+$, we use $\mathbb Z(m)$ for the finite cyclic group of order $m$. For a prime $p$ the symbol $\mathbb J_p$ is used for the group of $p$-adic integers and $\Z(p^\infty)$ for the Pr\"ufer group. Moreover, the symbol $\cont$ stands for the cardinality of the continuum.

Let $G$ be an abelian group. 
For a set $I$ we denote by $G^I$ the direct product $\prod_{i\in I}G$ and by $G^{(I)}$ the direct sum $\bigoplus_{i\in I} G$. For $x=(x_i)_{i\in I}\in G^{I}$ the support of $x$ is $\supp(x)=\{i\in I:x_i\neq0\}$. The subgroup of torsion elements of $G$ is $t(G)$, while for a prime $p$ we denote by $t_p(G)$ the $p$-component of $G$. Moreover, $D(G)$ denotes the divisible hull of $G$ and $d(G)$ the maximum divisible subgroup of $G$. We denote by $r_0(G)$ the torsion-free rank of $G$ and, for a prime $p$, $r_p(G)$ denotes the $p$-rank of $G$, that is, $\dim_{\mathbb F_p}(G[p])$, where $G[p]$ is the $p$-socle of $G$. More in general, for $k\in\N_+$, $G[k]=\{x\in G:k x=0\}$. 
Moreover, $\End(G)$ is the ring of all endomorphisms of $G$. 
For $\phi\in\End(G)$, we say that a subgroup $H$ of $G$ is $\phi$-invariant if $\phi(H)\subseteq H$.
For $k\in\Z$ we denote by $\mu_k:G\to G$ the multiplication by $k$, that is, $\mu_k(x)=k\cdot x$ for every $x\in G$.

For undefined terms see \cite{F}.

\subsection*{Aknowledgements}

We thank Professor Dikranjan for his very useful comments and suggestions, and also Professor Salce and Professor G\"obel for their suggestions about already existing examples and results, which helped us in completing this paper. We are also grateful to Professor Goldsmith for giving us preprints of \cite{GG} and \cite{GG2}, and to the referee, who gave us the possibility to improve the exposition of our results with many useful observations.

\section{First results and examples}\label{examples}

In this section we calculate the string numbers of many examples of abelian groups and state first results that have interest on their own and that will be applied in the following sections.

\medskip
For an endomorphism $\f$ of an abelian group $G$, $s(\f)=ns(\f)+s_0(\f)$ (see \cite[Equation 1.1]{DGV}); this formula gives immediately its counterpart for the string numbers of $G$, that is,
\begin{equation}\label{s=ns+s0}
s(G)=ns(G)+s_0(G);
\end{equation}
this is equivalent to $\mathfrak S=\Sns \cap \S0$.

\smallskip
Since a string is an infinite set, $s(F)=0$ for every finite abelian group $F$. So a first observation about the class $\mathfrak S$ is that it contains all finite abelian groups. 

The next examples follow respectively from \cite[Examples 3.16, 3.17, 3.18]{DGV} and from the values of the string numbers of the Bernoulli shifts collected in Table \ref{bernoulli}.

\begin{example}\label{Z}
\begin{itemize}
\item[(a)] $s(\Z)=0$.
\item[(b)] $s(\Z^2)=ns(\Z^2)=\infty$. 
\item[(c)] $s(\Q)=ns(\Q)=\infty$ and $s_0(\Q)=0$.
\item[(d)] If $G$ is a non-trivial abelian group, then $s(G^{(\N)})=ns(G^{(\N)})=s_0(G^{(\N)})=\infty$.
\end{itemize}
\end{example}


Given an abelian group $G$ and an endomorphism $\phi$ of $G$, let
$\F_{\phi}$ be the family of all the subgroups $H$ of $G$ such that
$\phi(H)=H$ and consider the partial order by inclusion on
$\F_{\phi}$. Let $\sc \phi=\sum_{H\in \F_{\phi}}H $. Since it is not
difficult to prove that $\sc \phi$ is an element of $ \F_{\phi}$, $\sc
\phi$ is a maximum for the poset $(\F_{\phi},\subseteq)$. In
particular, $\sc \phi$ is a $\phi$-invariant subgroup of $G$ such that
the endomorphism induced by $\phi$ on
 $\sc \phi$ is surjective. Furthermore,
$\sc \phi$ contains all the subgroups of $G$ with these properties. We
call  $\sc \phi$ the {\em surjective core} of $\phi$ (see \cite{DGV}
for an alternative description).

Consider now a pseudostring $S$ of $\phi$. Then $T(\phi, \langle
S\rangle)=\sum_{n\in\N}\phi^n\langle S\rangle$ is an element of
$\F_{\phi}$. This shows that every pseudo-string of $\phi$ is contained in
$\sc \phi$. In particular we obtain the following lemma, that permits to consider surjective endomorphisms in the computation of the string numbers of endomorphisms of abelian groups.

\begin{lemma}\label{s-c}
Let $G$ be an abelian group and $\phi\in\End(G)$. Then $s(\phi)=s(\phi\restriction_{\sc\phi})$, $ns(\phi)=ns(\phi\restriction_{\sc\phi})$ and $s_0(\phi)=s_0(\phi\restriction_{\sc\phi})$.
\end{lemma}

Let $G$ be an abelian group and $\phi\in\End(G)$. Recall that an element $x\in G$ is a \emph{periodic} point of $\phi$ if there exists $n\in\N_+$ such that $\phi^{n}(x)=x$; we denote by $\Per(\phi)$ the set of all periodic points of $\phi$. Moreover, an element $x\in G$ is a \emph{quasi-periodic} point of $\phi$ if there exist $n>m$ in $\N$ such that $\phi^{n}(x)=\phi^{m}(x)$; we denote by $Q\Per(\phi)$ the set of all quasi-periodic points of $\phi$. 

\smallskip
The next lemma, whose proof is a straightforward application of the definitions, gives equivalent conditions for a string to be singular.

\begin{lemma}\label{sing}
Let $G$ be an abelian group, $\phi\in\End(G)$ and $S=\{x_n\}_{n\in\N}$ a string of $\phi$. The following conditions are equivalent:
\begin{itemize}
\item[(a)] $S$ is singular;
\item[(b)] $x_0$ is quasi-periodic (i.e., $\{\phi^n(x_0):n\in\N\}$ is finite);
\item[(c)] $x_n$ is quasi-periodic for every $n\in\N$.
\end{itemize}
\end{lemma}

It is now possible to prove the following result, which characterizes the abelian groups endomorphisms with one of the string numbers zero in terms of their surjective core. 

\begin{proposition}\label{AAA}
Let $G$ be an abelian group and $\phi\in\End(G)$. Then:
\begin{itemize}
\item[(a)] $s(\phi)=0$ if and only if $\Per(\phi\restriction_{\sc\phi})=\sc\phi$;
\item[(b)] $ns(\phi)=0$ if and only if $Q\Per(\phi\restriction_{\sc\phi})=\sc\phi$;
\item[(c)] $s_0(\phi)=0$ if and only if $\ker\phi\restriction_{\sc\phi}=\ker\phi\cap\sc\phi=0$.
\end{itemize}
\end{proposition}
\begin{proof}
By Lemma \ref{s-c} we can suppose without loss of generality $\phi$ to be  surjective, that is, $\sc\phi=G$. For a fixed (but arbitrary) $x_0\in G\setminus \{0\}$, by the surjectivity of $\phi$, there exists a pseudostring $S_0=\{x_n\}_{n\in\N}$ of $\phi$. 

\smallskip
(a) Assume that $G=\Per(\phi)$. If $S=\{y_n\}_{n\in\N}$ is an arbitrary pseudostring of $\phi$, then $y_n$ is periodic for every $n\in\N$, so $O=\{\phi^k(y_0):k\in\N\}$ is finite and $y_n\in O$ for all $n\in\N$. This shows that $S\subseteq O$ is finite and so it cannot be a string. Hence, $s(\phi)=0$. To prove the converse implication, suppose that $G\neq\Per(\phi)$ and let $x_0\in G\setminus \Per(\phi)$. We show that $S_0$ is a string and so $s(\phi)=\infty$. In fact, if $x_n=x_m$ for some $n\geq m\in\N$, then $\phi^{n-m}(x_0)=\phi^m(x_n)=\phi^m(x_m)=x_0$. Since $x_0$ is not periodic, this implies $m=n$. Hence the elements of $S_0$ are pairwise distinct and so $S$ is a string.

\smallskip
(b) Suppose that $G=Q\Per(\phi)$. If $S$ is an arbitrary string of $\phi$, then $S$ has to be singular by Lemma \ref{sing}, and this proves $ns(\phi)=0$. To prove the converse implication, assume that $G\neq Q\Per(\phi)$ and let $x_0\in G\setminus Q\Per(\phi)$. In particular $x_0\in G\setminus \Per(\phi)$ and it suffices to proceed as in (a) to show that $S_0$ is a string. By Lemma \ref{sing}, $S_0$ is non-singular, and hence $ns(\phi)=\infty$.

\smallskip
(c) Assume $\phi$ to be injective. If $S=\{y_n\}_{n\in\N}$ is an arbitrary string of $\phi$, then $\phi^n(y_0)\neq 0$ for all $n\in \N$. So $S$ cannot be a null string, and this yields $s_0(\phi)=0$. To prove the converse implication, suppose that $\ker\phi\neq 0$ and let $x_0\in \ker\phi\setminus\{0\}$. Clearly, $x_0$ is not periodic and so it suffices to proceed as in (a) to show that $S_0$ is a string. Finally, $S_0$ is a null string since $\phi(x_0)=0$, hence $s_0(\phi)=\infty$. 
\end{proof}

Proposition \ref{AAA}(c) implies that $s_0(G)=0$ if $\phi$ is injective for every non-zero $\f\in\End(G)$. 
In particular, $s_0(\mathbb J_p)=0$.

%

Another way to use Proposition \ref{AAA}(c) to find groups belonging to $\S0$ was suggested us by the referee. Indeed, it is possible to represent a suitable ring $A$, {\em without zero-divisors}, as the endomorphism ring of an abelian group $G$ in such a way that all the endomorphisms of $G$ are injective. In the following example we give an idea of how this can be done.

\begin{example}
Let $A$ be a unitary and associative ring without zero-divisors. Since realization theorems are considerably easier when $A$ is countable,  we distinguish between the countable and uncountable case.
\begin{itemize}
\item[(a)] If $(A,+)$ is a countable, reduced and torsion-free abelian group, we can apply Corner's realization theorem (see \cite[Theorem A]{C1}) to find an abelian group $G$ (countable, reduced and torsion-free) such that $\End(G)\cong A$. Furthermore, since $A$ has no zero-divisors, by \cite[Lemma 1]{C} all $\phi\in\End(G)$ are injective. 
\item[(b)] In the non-countable case, the classical results of \cite{C1} can be substituted by some realization theorem based on Shelah's Black Box (see for example \cite{CG} or \cite[Section 12]{GT}). In particular, if $(A,+)$ is a cotorsion-free abelian group (i.e., $(A,+)$ does not contain any copy of $\Q$, $\Z(p)$, or $\J_p$ for every prime $p$), then there exist cotorsion-free abelian groups $G$ of arbitrarily large order such that $\End(G)\cong A$ (just take $A=\Z\setminus\{0\}$ in \cite[Theorem (6.3)]{CG}). Finally, using arguments similar to those in the proof of \cite[Theorem 2.11]{GS}, one can show that for such $G$ all $\phi\in\End(G)$ are injective.
\end{itemize}
\end{example}

The following property of the string numbers follows directly from Fact \ref{ent*}(e).

\begin{lemma}\label{LemmaY}
If $G=G_1 \oplus G_2$ for some subgroups $G_1,G_2$ of $G$, then: 
\begin{itemize}
\item[(a)] $s(G)=0$ implies $s(G_1)=s(G_2) =0$; 
\item[(b)] $ns(G)=0$ implies $ns(G_1)=ns(G_2) =0$;
\item[(c)] $s_0(G)=0$ implies $s_0(G_1)=s_0(G_2) =0$. 
\end{itemize}
\end{lemma}

The implications in items (a) and (b) of this lemma cannot be reversed, as $s(\Z) =ns(\Z)=0$, but $s(\Z^2)=ns(\Z^2)=\infty$ by Example \ref{Z}(a,b). It is less easy to see that the converse of item (c) does not hold true; Example \ref{hopf+hopf-non-hopf} will give a counterexample. 
 However, using Proposition \ref{AAA}, we can prove a partial converse of Lemma \ref{LemmaY} in Lemma \ref{LemmaY-fi}, but we need first to study the behavior of the surjective core on direct sums and products.

\begin{lemma}\label{sc-sp}
Let $G$ be an abelian group, $\phi \in \End(G)$ and let $H_\alpha$ be a $\phi$-invariant subgroup of $G$ for every $\alpha$ belonging to a set of indices $A$. 
\begin{itemize}
\item[(a)] If $G=\bigoplus_{\alpha\in A}H_\alpha$, then $\sc \phi=\bigoplus_{\alpha\in A}\sc \phi\restriction_{H_\alpha}$.
\item[(b)] If $G=\prod_{\alpha\in A}H_\alpha$, then $\sc \phi=\prod_{\alpha\in A}\sc \phi\restriction_{H_\alpha}$.
\end{itemize}
\end{lemma}
\begin{proof}
(a) Let $K$ be a subgroup of $G$. Suppose that $\phi (K)=K$ and let $K_\alpha=K\cap H_\alpha$, for every $\alpha\in A$. Since $\phi$ acts componentwise, $\phi\restriction_{H_\alpha}(K_\alpha)=K_\alpha$, for every $\alpha\in A$. This shows that $K\subseteq \bigoplus_{\alpha\in A}K_\alpha\subseteq \bigoplus_{\alpha\in A}\sc \phi \restriction_{H_\alpha}$, and in particular $\sc \phi\subseteq\bigoplus_{\alpha\in A}\sc \phi\restriction_{H_\alpha}$.

On the other hand, let $K_\alpha\subseteq H_\alpha$ be subgroups of $H_\alpha$ such that $\phi\restriction_{H_\alpha}(K_\alpha)=K_\alpha$ for every $\alpha\in A$. Then also $\phi(\bigoplus_{\alpha\in A} K_\alpha)=\bigoplus_{\alpha\in A} K_\alpha$, as $\phi$ acts componentwise. This shows that $\bigoplus_{\alpha\in A} K_\alpha\subseteq \sc \phi$, and in particular $\sc \phi\supseteq\bigoplus_{\alpha\in A}\sc \phi\restriction_{H_\alpha}$.

\smallskip
(b) Proceed analogously to (a).
\end{proof}

\begin{lemma}\label{LemmaY-fi}
Let $G$ be an abelian group such that $G=\bigoplus_{\alpha\in A}H_\alpha$, where $A$ is a set of indices and the $H_\alpha$ are fully invariant subgroups of $G$ for all $\alpha\in A$. Then:
\begin{itemize}
\item[(a)] $s(G)=0$ if and only if $s(H_\alpha)=0$ for every $\alpha\in A$; 
\item[(b)] $ns(G)=0$ if and only if $ns(H_\alpha)=0$ for every $\alpha\in A$;
\item[(c)] $s_0(G)=0$ if and only if $s_0(H_\alpha)=0$ for every $\alpha\in A$.
\end{itemize}
\end{lemma}
\begin{proof}
The ``only if" part of (a), (b) and (c) is a particular case of Lemma \ref{LemmaY}.  We now prove the ``if" part of (a); the ``if'' part of (b) and (c) can be proved similarly.  Suppose that $s(H_\alpha)=0$ for every $\alpha\in A$. Let $\phi\in\End(G)$; we have to show that $s(\phi)=0$. By our hypotheses, each $H_\alpha$ is a fully invariant subgroup of $G$ and so $\phi_\alpha=\phi\restriction_{H_\alpha}$ is an endomorphism of $H_\alpha$ for every $\alpha\in A$.  Let $x=(x_\alpha)_{\alpha\in A}\in G=\bigoplus_{\alpha\in A} H_\alpha$ with $x\in \sc \phi$.
Then $x_\alpha\in \sc \phi_\alpha$ for every $\alpha\in A$, by Lemma \ref{sc-sp}. By Proposition \ref{AAA}(a), since $s(H_\alpha)=0$, we get $x_\alpha\in \Per(\phi_\alpha)$ for every $\alpha\in A$. Therefore, $x\in \Per(\phi)$ and this yields $\sc\phi=\Per(\phi\restriction_{\sc\phi})$. Hence, Proposition \ref{AAA}(a) implies $s(\phi)=0$.
\end{proof}

We prove now that bounded abelian $p$-groups with one of the string numbers zero have to be finite:

\begin{lemma}\label{LemmaW}
If $G$ is an infinite bounded abelian $p$-group for some prime $p$, then $s(G)=ns(G)=s_0(G)=\infty$.
\end{lemma}
\begin{proof}
Since $G$ contains as a direct summand a subgroup isomorphic to $\Z(p^k)^{(\N)}$ for some $k\in\N_+$, apply Example \ref{Z}(d) and Lemma \ref{LemmaY} to conclude that $s(G)=ns(G)=s_0(G)=\infty$. 
\end{proof}

The following result solves Problem \ref{pb} for finitely generated abelian groups.

\begin{proposition}\label{fg}
Let $G$ be a finitely generated abelian group. Then: 
\begin{itemize}
\item[(a)] $s(G)=ns(G)$;
\item[(b)] $s(G)=0$ if and only if $r_0(G)\leq1$;
\item[(c)] $s_0(G)=0$.
\end{itemize}
\end{proposition}
\begin{proof}
(a,b) The case $r_0(G)=0$ is trivial, since $G$ is finite. 

Assume that $r_0(G)= 1$. We show that $s(G)=ns(G)=0$. In fact, let $\f\in\End(G)$; we have to prove that $s(\f)=0$.
Let $S=\{x_n\}_{n\in\N}$ be a pseudostring of $\f$. Since $G\cong\Z\oplus F$ for some finite subgroup $F$ of $G$, we can write in a unique way $x_n=z_n+f_n$ with $z_n\in\Z$ and $f_n\in F$. Moreover there exists $k\in\N_+$ such that $kF=0$. Since $S$ is a pseudostring of $\f$, we have that $kS=\{kz_n\}_{n\in\N}$ is a pseudostring of $\f\restriction_{kG}:kG\to kG$; note that $kG\cong \Z$ is a fully invariant subgroup of $G$. If $kS$ were infinite then, $kS$ would contain a string of $\phi\restriction_{kG}$ in view of \cite[Lemma 2.10]{DGV}; therefore, $s(kG)=\infty$ and this contradicts Example \ref{Z}(a), as $kG\cong\Z$. Hence $kS$ is finite and consequently $\{z_n\}_{n\in\N}$ is finite as well. Since $S$ is contained in $\{z_n\}_{n\in\N}\oplus F$, it follows that $S$ is finite, hence $S$ is not a string. 

Suppose that $r_0(G)>1$. Then $G$ has $\Z^2$ as direct summand and so $s(G)=ns(G)=\infty$ thanks to Example \ref{Z}(b) and Lemma \ref{LemmaY}(a,b).

\smallskip
(c) Let $\f\in\End(G)$. Since every subgroup of $G$ is finitely generated, by Lemma \ref{s-c} we can assume without loss of generality that $\phi$ is surjective. Then $\phi$ is also injective and so $s_0(\phi)=0$ by Proposition \ref{AAA}(c).
\end{proof}



As a corollary of Proposition \ref{fg} we obtain the following characterization of free abelian groups with string numbers zero.

\begin{corollary}\label{free}
Let $G$ be a free abelian group. Then:
\begin{itemize}
\item[(a)] $s(G)=ns(G)$;
\item[(b)] $s(G)=0$ if and only if $r_0(G)\leq 1$;
\item[(c)] $s_0(G)=0$ if and only if $r_0(G)$ is finite.
\end{itemize}
\end{corollary}
\begin{proof}
(a,b) If $r_0(G)\geq 1$, then $G$ has $\Z^2$ as a direct summand. By Example \ref{Z}(b) $ns(\Z^2)=s(\Z^2)=\infty$ and so Lemma \ref{LemmaY}(a,b) gives $ns(G)=s(G)=\infty$. If $r_0(G)\leq 1$, $G=\Z$ and $ns(G)=s(G)=0$ by Example \ref{Z}(a).

\smallskip
(c) If $r_0(G)$ is finite, then $G$ is finitely generated and so Proposition \ref{fg} gives $s_0(G)=0$. If $r_0(G)$ is infinite, then $\Z^{(\N)}$ is a direct summand of $G$. Hence $s_0(G)=\infty$ by Example \ref{Z}(d) and Lemma \ref{LemmaY}(c).
\end{proof}

The next proposition gives a sufficient condition for an abelian group $G$ to have non-singular string number zero. 

\begin{proposition}\label{U->ns=0}
Let $G$ be an abelian group such that $G=\bigcup_{n\in\N} G_n$, where $G_n$ is a finite fully invariant subgroup of $G$ for every $n\in\N$. Then $ns(G)=0$.
\end{proposition}
\begin{proof}
Let $\f\in\End(G)$ and let $S=\{x_n\}_{n\in\N}$ be a string of $\phi$. We show that $S$ is singular. Indeed, there exists $n\in\N$ such that $x_0\in G_n$. Since $G_n$ is finite and fully invariant, $\{\phi^n(x_0):n\in\N\}$ is finite as well, as it is contained in $G_n$, and so $S$ is singular by a pigeon-hole argument. 
\end{proof}

The converse implication of this proposition does not hold true; indeed, we shall see that the abelian group $G$ given in Example \ref{hopf-s0=infty} is not countable, yet $ns(G)=0$, while every abelian group $G$ satisfying the hypotheses of Proposition \ref{U->ns=0} is necessarily torsion and countable. 

\smallskip
As a consequence of Proposition \ref{U->ns=0} we obtain the following example.

\begin{example}\label{prufer}
For every prime $p$ and every $n\in\N_+$, we have that:
\begin{itemize}
\item[(a)] $s_0((\Z(p^\infty))^n)=s((\Z(p^\infty))^n)=\infty$ by \cite[Example 3.19]{DGV} and Lemma \ref{LemmaY}(a,c), and $ns(\Z(p^\infty)^n)=0$ by Proposition \ref{U->ns=0};
\item[(b)] $s_0((\Q/\Z)^n)=s((\Q/\Z)^n)=\infty$ by (a) and Lemma \ref{LemmaY}(a,c), and $ns((\Q/\Z)^n)=0$ by Proposition \ref{U->ns=0}.
\end{itemize}
\end{example}

The next lemma will be covered by Theorem \ref{div-th}, which characterizes the divisible abelian groups with string numbers zero. We state it here because it will be useful in what follows.

\begin{lemma}\label{div}
Let $D$ be a non-trivial divisible abelian group.
\begin{itemize}
\item[(a)] If $D$ is torsion, then $s_0(D)=\infty$. 
\item[(b)] If $D$ is torsion-free, then $ns(D)=\infty$.
\end{itemize}
In particular, $s(D)=\infty$. 
\end{lemma}
\begin{proof}
(a) Since $D$ has $\Z(p^\infty)$ as a direct summand for some prime $p$, and $s_0(\Z(p^{\infty}))=\infty$ by Example \ref{prufer}(a), Lemma \ref{LemmaY}(c) yields $s_0(D)=\infty$.

\smallskip
(b) Now $D$ has $\Q$ as a direct summand and $ns(\Q)=\infty$ by Example \ref{Z}(c); hence Lemma \ref{LemmaY}(b) gives $ns(D)=\infty$.

\smallskip
Since $D\cong t(D)\oplus D/t(D)$, and at least one of $t(D)$ and $D/t(D)$ is non-trivial, we can conclude that $s(D)=\infty$ by (a), (b) and Lemma \ref{LemmaY}(a).
\end{proof}

As a clear consequence of Lemmas \ref{div} and \ref{LemmaY}, we have that an abelian group $G$ is reduced if one of the following conditions holds: 
\begin{itemize}
\item[(a)] $s(G) =0$;
\item[(b)] $G$ is torsion and $s_0(G)=0$;
\item[(c)] $G$ is torsion-free and $ns(G)=0$.
\end{itemize}

\section{Torsion abelian groups with no strings}\label{s=0-tor}

In this section we consider Problem \ref{pb} for torsion abelian groups. In particular, Theorem \ref{ThB} shall solve it for the string number and the null string number, characterizing all torsion abelian groups in $\mathfrak S$ and in $\S0$. Moreover, we will provide sufficient and necessary conditions for a torsion abelian group to belong to $\Sns$.

\medskip
The next lemma is a clear consequence of Lemma \ref{LemmaY-fi}. In fact, every torsion abelian group $G$ is the direct sum of its $p$-components
$t_p(G)$, that are fully invariant in $G$. We state it  explicitly as it shows that working with the string numbers of torsion abelian groups it is possible to reduce to the case of abelian $p$-groups.

\begin{corollary}\label{red-to-p}
Let $G$ be a torsion abelian group. Then:
\begin{itemize}
\item[(a)] $s(G)=0$ if and only if $s(t_p(G))=0$ for every prime $p$;
\item[(b)] $ns(G)=0$ if and only if $ns(t_p(G))=0$ for every prime $p$;
\item[(c)] $s_0(G)=0$ if and only if $s_0(t_p(G))=0$ for every prime $p$.
\end{itemize}
\end{corollary}

The following lemma characterizes the divisible torsion abelian groups in $\mathfrak S$, $\Sns$ and $\S0$.

\begin{lemma}\label{torsion-div}
Let $p$ be a prime and let $D$ be a divisible abelian $p$-group. Then:
\begin{itemize}
\item[(a)] $s(D)=0$ if and only if $s_0(D)=0$ if and only if $D=0$;
\item[(b)] $ns(D)=0$ if and only if $r_p(D)$ is finite.
\end{itemize}
\end{lemma}
\begin{proof}
(a) Follows from Lemma \ref{div}.

\smallskip
(b) Since $D\cong\Z(p^\infty)^{(\alpha_p)}$ for some cardinal $\alpha_p$, if $r_p(D)=\alpha_p$ is infinite, then $ns(D)=\infty$ by Example \ref{Z}(d) and Lemma \ref{LemmaY}(b). If $r_p(D)$ is finite, then $ns(D)=0$ by Example \ref{prufer}(a).
\end{proof}

The following lemma shows that, in order to calculate one of the string numbers of an abelian $p$-group $G$, we can always suppose $G$ to be either divisible or reduced. 

\begin{lemma}\label{ridotti}
Let $p$ be a prime and $G$ an abelian $p$-group. We can write $G=d(G)\oplus R$ where $d(G)$ is the maximum divisible subgroup of $G$ and $R\cong G/t(G)$ is  reduced. Then:
\begin{itemize}
\item[(a)] $s(G)=s(d(G))+s(R)$;
\item[(b)] $ns(G)=ns(d(G))+ns(R)$;
\item[(c)] $s_0(G)=s_0(d(G))+s_0(R)$.
\end{itemize}
\end{lemma}
\begin{proof}
(a) Follows from (b) and (c) using \eqref{s=ns+s0}.

\smallskip
(b) By Lemma \ref{LemmaY}(b) we have that $ns(G)\geq ns(d(G))+ns(R)$, therefore we need to prove only that $ns(d(G))=ns(R)=0$ implies $ns(G)=0$. So suppose that $ns(d(G))=ns(R)=0$. By Lemma \ref{torsion-div} the $p$-rank of $d(G)$ has to be finite. Consider an endomorphisms $\phi:G\to G$ and suppose, looking for a contradiction, that $S=\{x_n\}_{n\in\N}$ is a non-singular string of $\phi$. For every $n\in\N$, we can write uniquely $x_n=d_n+c_n$ with $d_n\in d(G)$ and $c_n\in R$. Furthermore, since $d(G)$ is fully invariant in $G$, we can represent $\phi$ as a matrix $\left(\begin{smallmatrix}\phi_d &\phi_{rd}\\ 0&\phi_r\end{smallmatrix}\right):d(G)\oplus R\to d(G)\oplus R$. 
We verify that $ns(\phi_r)=\infty$, and this will provide the contradiction we are looking for. 

Let $S'=\{c_n\}_{n\in\N}$. It is clear that $S'$ is a pseudostring of $\phi_r$ and so $c_0\in \sc \phi_r$. Consider now the trajectory $T(\phi,\langle x_0\rangle)$, that is an infinite set by our assumption that $S$ is non-singular. On the other hand, let $k$ be the order of $x_0$, then 
$$T(\phi,\langle x_0\rangle)\subseteq d(G)[k]\oplus T(\phi_r,\langle c_0\rangle).$$
Finally notice that $d(G)[k]$ is finite by our assumption on the $p$-rank of $d(G)$, and so $T(\phi_r,\langle c_0\rangle)$ has to be infinite. This is equivalent to say that $c_0\in \sc \phi_r\setminus Q\Per (\phi_r)$. We can now conclude by Lemma \ref{AAA}(b).

\smallskip
(c) Follows from Lemma \ref{LemmaY}(c) and Lemma \ref{div}.
\end{proof}

Since the string numbers of torsion divisible abelian groups are computed in Lemma \ref{torsion-div}, we have that Corollary \ref{red-to-p} and Lemma \ref{ridotti} give the possibility to consider only reduced abelian $p$-groups in the computation of the string numbers of torsion abelian groups.

\medskip
We recall that an abelian $p$-group $G$ is \emph{separable} if $p^\omega G=0$. Consider now, for a prime $p$, two separable abelian $p$-groups $G$ and $H$. A homomorphism $\phi:G\to H$ is \emph{small} if 
\begin{equation*}\label{small}
\forall   k\in\N,  \  \ \exists n\in \N,  \text{ such that } (e(x)\geq n)\Rightarrow (e(\phi(x))\leq e(x)-k)\ \ \forall x\in G;
\end{equation*}
where $e(-)$ denotes the exponent of an element. In other words, $\phi$ is small if for every $k\in\N$, there exists $n\in\N$ such that $\phi((p^nG)[p^k])=0$. We will use the properties of small endomorphisms in the proof of the next theorem, which solves completely Problem \ref{pb} for the string number and the null string number in the case of torsion abelian groups, in view of the previous reduction to reduced abelian $p$-groups.

\begin{theorem}\label{ThB}
Let $p$ be a prime and let $G$ be a reduced abelian $p$-group.
Then the following conditions are equivalent:
\begin{itemize}
\item[(a)]$s(G)=0$;
\item[(b)]$s_0(G)=0$;
\item[(c)] $G$ is finite. 
\end{itemize}
In particular, $s(G)=s_0(G)$. 
\end{theorem}
\begin{proof}
(a)$\Rightarrow$(b) Is trivial.

\smallskip
(b)$\Rightarrow$(c) 
Assume that $G$ is infinite. Hence $G$ is unbounded by Lemma \ref{LemmaW}. Let $B$ be a basic subgroup of $G$; $B$ must then be unbounded as well. Then there exist  two subgroups $B'$ and $C$ of $B$ such that $B=B'\oplus C$, where $B'=\bigoplus_{n\in\N}\Z(p^{k_n})$ with $0<k_0<k_1<\ldots<k_n<\ldots$ in $\N$. 
Denote by $e_n=(0,\ldots,0,1_{\Z(p^{k_n})},0,\ldots)$ the $n$-th standard generator of $B'$, for all $n\in\N$.
Define the endomorphism $\phi_{B'}$ of $B'$ on the generators, by letting
$$e_1\mapsto 0,\ \ e_{m!}\mapsto e_{(m-1)!}\ \text{ for every $m\geq 2$}, \ \ e_n\mapsto 0\ \text{ if }\  n\notin \{k!:k\geq 2\}.$$
Then $\phi_{B'}$ is a small endomorphism of $B'$ and it can be extended to a small endomorphism $\phi_B$ of $B$ simply by letting $\phi_B(C)=0$. Since $B$ is pure in $G$ and $\phi_B$ is small, we can find an endomorphism $\phi$ of $G$ extending $\phi_B$ (see \cite[Theorem 4.4]{P}). Now it is clear that $\{e_{n!}\}_{n\in\N_+}$ is a null-string of $\phi$ and so $s_0(\phi)=\infty$. Therefore, $s_0(G)=\infty$. 

\smallskip
(c)$\Rightarrow$(a) is clear. 
\end{proof}

If a reduced abelian $p$-group $G$ is finite, then obviously $ns(G)=0$. Together with Lemma \ref{torsion-div}(b) and the reduction to reduced abelian $p$-groups provided by Corollary \ref{red-to-p}(b) and Lemma \ref{ridotti}(b), this trivial observation gives the next corollary showing a sufficient condition for a torsion abelian group to be in $\Sns$. Example \ref{hopf-s0=infty} will show that this condition is not necessary.

\begin{corollary}\label{B-2}
Let $G$ be a torsion abelian group. If $r_p(G)$ is finite for every prime $p$, then $ns(G)=0$.
\end{corollary}
The problem of the characterization of the torsion abelian groups $G$ with $ns(G)=0$ is open, as stated by Problem \ref{Ques1}. Nevertheless, we have some reductions, as the following

\begin{lemma}\label{UK-finite}
Let $p$ be a prime and let $G$ be a reduced abelian $p$-group. If $ns(G)=0$, then the finite Ulm-Kaplansky invariants of $G$ are finite. 
\end{lemma}
\begin{proof}
Assume that some of the finite Ulm-Kaplansky invariants $\alpha_n$ of $G$ is infinite. Then $G$ has a direct summand of the form $H=\Z(p^n)^{(\N)}$. Since $ns(H)=\infty$ by Example \ref{Z}(d), Lemma \ref{LemmaY}(b) gives $ns(G)=\infty$.
\end{proof}

Thanks to this lemma it is possible to see that ``large'' torsion abelian group are not in $\Sns$:

\begin{theorem}\label{||<c}
Let $G$ be a torsion abelian group. If $ns(G)=0$, then $|G|\leq\cont$.
\end{theorem}
\begin{proof}
We can assume without loss of generality that $G$ is a $p$-group, that is, $G=t_p(G)$, for some prime $p$. Indeed, $ns(t_p(G))=0$ for every prime $p$ by Corollary \ref{red-to-p}(b) and $|t_p(G)|\leq\cont$ for every prime $p$ implies $|G|\leq\cont$.
Now the hypothesis $ns(G)=0$ implies that the maximum divisible subgroup $d(G)$ of $G$ has $ns(d(G))=0$ as well, in view of Lemma \ref{LemmaY}(b). By Theorem \ref{ThB} $r_p(d(G))$ is finite, so in particular $|d(G)|=\aleph_0$. Thus we can assume that $G$ is reduced.
By Lemma \ref{UK-finite}, the finite Ulm-Kaplansky invariants $\alpha_n$ of $G$ are finite. Now \cite[Theorem 34.3]{F} gives $|G| \leq \cont$.
\end{proof} 

Let $G$ be a torsion abelian group, $\f\in\End(G)$ and let $S$ be a non-singular string of $\f$. Clearly the $\f$-trajectory $T(\f,\langle x\rangle)$ is infinite for every $x\in S$, and so $\ent(\f)>0$ by \cite[Proposition 2.4]{DGSZ}.
In particular,
\begin{equation}\label{ent=0->ns=0}
\ent(G)=0\Longrightarrow ns(G)=0.
\end{equation}

Looking at the implication in \eqref{ent=0->ns=0}, the following natural question was suggested by the referee and remains open. 

\begin{problem}
Does there exist an abelian $p$-group $G$ with $ns(G)=0$ but $\ent(G)=\infty$?
\end{problem}

Equation \eqref{ent=0->ns=0} is equivalent to say that the class of torsion abelian groups with algebraic entropy zero is contained in $\Sns \cap \mathfrak T$.
So the results from \cite[Section 5]{DGSZ} on torsion abelian groups with algebraic entropy zero hold also for torsion abelian groups with non-singular string number zero. In particular we have the following theorem, which is the counterpart of \cite[Theorem 5.18]{DGSZ}, based on a Corner's realization theorem from \cite{C1}.

\begin{theorem}\label{th1}
Given an ordinal $\lambda<\omega^2$, there exists a family of $2^{2^{\omega}}$ abelian $p$-groups of length $\lambda$ and with non-singular string number zero. 
\end{theorem}

We discuss now the case when the non-singular string number is infine, starting from the following example.

\begin{example}\label{p-basic}
For a prime $p$, the standard $p$-basic subgroup $B_p=\bigoplus_{n\in\N_+}\Z(p^n)$ has $ns(B_p)=\infty$.
 
To show this, we construct an endomorphism $\phi$ of $B_p$ admitting a non-singular string. Let $e_n=(0,\ldots,0,1_{\Z(p^n)},0,\ldots)$ be the $n$-th standard generator of $B_p$, for every $n\in\N_+$. Define $\phi$ on the generators of $B_p$ letting, for every $n\in\N_+$,
$$e_{2(n+1)}\mapsto e_{2n},\ \ \ \ e_2\mapsto e_1,\ \ \ \ e_{2n-1}\mapsto p^2\cdot e_{2n+1}.$$
Then clearly $S=\{e_{2n}\}_{n\in\N_+}$ is a non-singular string of $\f$. 
\end{example}

This example can be generalized to any abelian group of the form $\bigoplus_{n\in N}\Z(p^n)$, where $N$ is an infinite subset of $\N$, and so also to abelian groups of the form $\bigoplus_{n\in\N_+}\Z(p^n)^{(\alpha_n)}$, where $\alpha_n>0$ for infinitely many $\alpha_n$, thanks to Lemma \ref{LemmaY}(b).

\smallskip
An abelian $p$-group is said to be \emph{torsion-complete} if it is the torsion part of the $p$-adic completion of a direct sum of cyclic abelian $p$-groups. Torsion-complete abelian $p$-groups are necessarily separable. Even if an abelian $p$-group $G$ has an unbounded $p$-basic subgroup $B_p$, it is not always possible to extend the endomorphism of $B_p$ given in Example \ref{p-basic} to the whole group $G$. But, if $G$ is torsion-complete, then every endomorphism of $B_p$ extends to $G$ (see \cite[Section 68]{F}), and so Example \ref{p-basic} permits to prove the following 

\begin{theorem}\label{tcomp}
Let $p$ be a prime, and let $G$ be an unbounded torsion-complete abelian $p$-group. Then $s_0(G)=ns(G)=s(G)=\infty$. 
\end{theorem}
\begin{proof}
By Theorem \ref{ThB}, $s_0(G)=s(G)=\infty$.
Since $G$ is separable, in particular $G$ is reduced. Then $G$ contains an unbounded $p$-basic subgroup $B_p$. Moreover, $ns(B_p)=\infty$ by Example \ref{p-basic} and each endomorphism of $B_p$ extends to an endomorphism of $G$, since $G$ is torsion-complete. Hence $ns(G)=\infty$.
\end{proof}

In \cite[Section 4]{DGSZ} torsion abelian groups with infinite algebraic entropy are produced. To exhibit these groups the authors show that they have a direct summand which is an infinite direct sum of cyclic $p$-groups. So the same groups have also infinite string numbers in view of Example \ref{Z}(d) and Lemma \ref{LemmaY}. In particular, the following is the counterpart of \cite[Theorem 4.5]{DGSZ}.

\begin{theorem}\label{th2}
The reduced abelian $p$-groups which are either totally projective or $p^{\omega+1}$-projective have infinite string numbers.
\end{theorem}

\section{The torsion-free case}\label{tf-sec}

In this section we study Problem \ref{pb} in the case of torsion-free abelian groups, starting from the following

\begin{lemma}\label{s0(fin-rank)=0}
Let $G$ be a torsion-free abelian group of finite rank. Then $s_0(G)=0$.
\end{lemma}
\begin{proof}
Let $\phi\in\End(G)$. Since every subgroup of $G$ is torsion-free and of finite rank, by Lemma \ref{s-c} we can assume without loss of generality that $\phi$ is surjective. Then $\phi$ is injective and hence $s_0(\phi)=0$ by Proposition \ref{AAA}(c).
\end{proof}

The following theorem characterizes completely the divisible abelian groups in $\mathfrak S$, $\Sns$ and $\S0$.

\begin{theorem}\label{div-th}
Let $D$ be a divisible abelian group. Then:
\begin{itemize}
\item[(a)] $s(D)=0$ if and only if $D=0$;
\item[(b)] $ns(D)=0$ if and only if $D=t(D)$ and $r_p(D)$ is finite for every prime $p$;
\item[(c)] $s_0(D)=0$ if and only if $t(D)=0$ and $r_0(D)$ is finite.
\end{itemize}
\end{theorem}
\begin{proof}
Since $D$ is divisible, $D\cong t(D)\oplus D/t(D)$.

\smallskip
(a) Follows from the last statement of Lemma \ref{div}.

\smallskip
(b) If $ns(D)=0$, then $D=t(D)$ by Lemma \ref{div}(b), $ns(t_p(D))=0$ for every prime $p$ by Corollary \ref{red-to-p}(b), and
$r_p(D)$ finite for every prime $p$ by Lemma \ref{torsion-div}.
If $D=t(D)$ and $r_p(D)$ is finite for every prime $p$, then $ns(D)=0$ by Corollary \ref{red-to-p}(b) and Lemma \ref{torsion-div}.

\smallskip
(c) If $s_0(D)=0$, then $t(D)=0$ by Lemma \ref{div}(a). If $r_0(D)$ is infinite, then $D$ contains $\Q^{(\N)}$ as a direct summand; since $s_0(\Q^{(\N)})=\infty$ by Example \ref{Z}(d), hence $s_0(D)=\infty$ by Lemma \ref{LemmaY}(c).
If $t(D)=0$ and $r_0(D)=n$ is finite, $D\cong \Q^n$ and $s_0(D)=0$ by Lemma \ref{s0(fin-rank)=0}.
\end{proof}

The next proposition gives a necessary condition for a torsion-free abelian group to belong to $\Sns$.

\begin{proposition}\label{tf-pomega}
Let $G$ be a torsion-free abelian group. If $ns(G)=0$, then $p^\omega G=0$ for every prime $p$.
\end{proposition}
\begin{proof}
Assume that $p^\omega G\neq 0$ for some prime $p$. Since $p^\omega G=\sc\mu_p$, it follows from \cite[Lemma 3.23 and Corollary 3.21(b)]{DGV} that $ns(\mu_p)=\infty$, hence $ns(G)=\infty$.
\end{proof}

This proposition implies in particular that $s(\mathbb J_p)=ns(\mathbb J_p)=\infty$.

\smallskip
The converse of Proposition \ref{tf-pomega} is clearly false in general. In fact, it suffices to take $G=\Z^2$, which is a torsion-free abelian group with $p^{\omega}G=0$ for every prime $p$, but $ns(G)=\infty$ by Example \ref{Z}(b). However, the converse implication of Proposition \ref{tf-pomega} holds for endorigid torsion-free abelian groups:

\begin{theorem}\label{endorigid}
If $G$ is an endorigid torsion-free abelian group, then $s_0(G)=0$.  In particular, $s(G)=ns(G)$, and
the following conditions are equivalent:
\begin{itemize}
\item[(a)] $s(G)=0$;
\item[(b)] $p^\omega G = 0$ for all primes $p$. 
\end{itemize}
\end{theorem}
\begin{proof}
Since $G$ is torsion-free, the non-zero endomorphisms of $G$ are injective. Hence $s_0(G)=0$ by Proposition \ref{AAA}(c), and $s(G)=ns(G)$ by \eqref{s=ns+s0}.

\smallskip
(a)$\Rightarrow$(b) Is given by Proposition \ref{tf-pomega}.

\smallskip
(b)$\Rightarrow$(a) Let $\f\in\End(G)=\Z$. Then $\f=\mu_k$ for some $k\in\Z$ and $\sc\f\subseteq p^\omega G$ if $p$ is a prime dividing $k$. Since $p^\omega G=0$ by hypothesis, $\sc\f=0$ as well. Hence $s(\phi)=0$ by Lemma \ref{s-c}.
\end{proof}

As a consequence of Proposition \ref{tf-pomega} and Theorem \ref{endorigid} we obtain a complete characterization of the torsion-free abelian groups of torsion-free rank $1$ in $\mathfrak S$.

\begin{corollary}\label{tf:rk1}
Let $G$ be a torsion-free abelian group of rank $1$, then $s(G)= 0$ if and only if no infinity appears in the type of $G$. 
\end{corollary}
\begin{proof}
Suppose that $s(G)=0$. By Proposition \ref{tf-pomega}, $p^{\omega}G=0$ for every prime $p$; this implies that no infinity appears in the type of $G$. 
On the other hand, since $G$ is isomorphic to a subgroup of $\Q$, $\End(G)$ is isomorphic to the subring of $\Q$ generated by $\left\{1,\frac{1}{p}:p\ \text{prime},\ pG=G\right\}$ (see \cite[Section 106, Example 4]{F}). This shows that, if there is no infinity in the type of $G$, then $\End(G)\cong \Z$ and Theorem \ref{endorigid} applies to conclude $s(G)=0$.
\end{proof}

Using Corollary \ref{tf:rk1} we can explicitly construct an example of a decomposable torsion-free abelian group $G$ such that $s(G)=0$:

\begin{example}
Let $A$ and $B$ be two non trivial subgroups of $\Q$ whose types are respectively $(a_1,\dots,a_n,\dots)$ and $(b_1,\dots,b_n,\dots)$ where $a_{2n}=b_{2n+1}=1$ and $a_{2n+1}=b_{2n}=0$ for every $n\in\N_+$. Then  $s(A)=s(B)=0$ in view of Corollary \ref{tf:rk1}, and both $A$ and $B$ are fully invariant subgroups of $A\oplus B$. Now Lemma \ref{LemmaY-fi}(a) yields $s(A\oplus B)=0$.
\end{example}

The following theorem shows that the torsion-free abelian groups are difficult to classify from the point of view of the string numbers. To prove it we use strong results from \cite{GT} and \cite{GF}.

\begin{theorem}\label{tf--}
For every infinite cardinal $\lambda$, there exist indecomposable torsion-free abelian groups $G$ of rank $\lambda$ with $ns(G)=s(G)=\infty$ and there exist indecomposable torsion-free abelian groups $H$ of rank $\lambda$ with $s(H)=0$.
\end{theorem}
\begin{proof}
Let $\lambda$ be an infinite cardinal. By \cite[Corollary 14.5.3]{GT} there exists an endorigid torsion-free abelian group $G$ with $r_0(G)=\lambda$, and such that $p^\omega G\neq 0$ for some prime $p$. By Theorem \ref{endorigid}, $ns(G)=s(G)=\infty$. 
By \cite[Corollary 5.4]{GF} there exists an endorigid torsion-free abelian group $H$ with $r_0(H)=\lambda$, and such that $p^\omega H=0$ for every prime $p$. By Theorem \ref{endorigid}, $s(H)=0$.
\end{proof}

\section{Closure properties for $\mathfrak S$, $\Sns$, and $\S0$}\label{stab-sec}

In this section we consider some closure properties of the classes $\mathfrak S$, $\Sns$ and $\S0$, namely, closure under taking subgroups, quotients, summands, direct sums, extensions and direct products.

\smallskip
By Theorem \ref{ThB}, if $\mathfrak T$ denotes the class of all torsion abelian groups,
$$\mathfrak S\cap \mathfrak T=\S0\cap\mathfrak T=\{G\ \text{torsion}:t_p(G)\ \text{finite, for every prime}\ p\}.$$
Since the property of having finite $p$-components is stable under taking subgroups, quotients, finite direct sums and extensions, the class $\mathfrak S\cap\mathfrak T$ is closed under taking subgroups, quotients, finite direct sums and extensions. 

\smallskip
Contrarily to the torsion case, the classes $\mathfrak S$, $\Sns$ and $\S0$ satify these properties only in very particular cases. 
Indeed, part (a) of the next example shows that the classes $\mathfrak S$, $\Sns$ and $\S0$ are not closed under taking subgroups. Moreover, item (b) shows that $\Sns\cap\mathfrak T$ is not closed under taking subgroups, contrarily to $\mathfrak S\cap \mathfrak T$.

\begin{example}\label{mono-sub}
\begin{itemize}
\item[(a)] Let $G$ be an endorigid torsion-free abelian group with $r_0(G)\geq\aleph_0$, such that $p^\omega G=0$ for every prime $p$ (as noted in the proof of Theorem \ref{tf--}, such a $G$ exists). Then $s(G)=ns(G)=s_0(G)=0$ in view of Theorem \ref{endorigid}, while $G$ contains a subgroup $H$ isomorphic to $\Z^{(\N)}$, which has $s(H)=ns(H)=s_0(H)=\infty$ by Example \ref{Z}(d).
\item[(b)] By Theorem \ref{th1} there exists an abelian $p$-group $G$ such that $ns(G)=0$, and with an infinite $p$-basic subgroup $B$, which has $ns(B)=\infty$ by Example \ref{p-basic}.
\end{itemize}
\end{example}

By Fact \ref{ent*}(c) we derive immediately that, given an abelian group $G$ and a fully invariant subgroup $H$ of $G$, then $s(G)\geq s(H)$, $ns(G)\geq ns(H)$ and $s_0(G)\geq s_0(H)$. Equivalently, $\mathfrak S$, $\Sns$ and $\S0$ are closed under taking fully invariant subgroups.

\medskip
Part (a) of the next example shows that the classes $\mathfrak S$ and $\S0$ are not closed under taking quotients. Moreover, part (b) shows that $\Sns\cap\mathfrak T$ (and so also $\Sns$) is not closed under taking quotients, contrarily to $\mathfrak S\cap \mathfrak T$.

\begin{example}\label{mono-subbis}
\begin{itemize}
\item[(a)] 
Consider the group $G$ of Example \ref{mono-sub}(a). Since $r_0(G)\geq\aleph_0$, we can find an embedding
$\Z^{(\Q)}\hookrightarrow G$. Moreover, the canonical surjection $\Z^{(\Q)}\to \Q$ can be extended to a surjection $G\to \Q$, as $\Q$ is divisible. Now choose a prime $p$; since $\Z(p^\infty)$ is a quotient of $\Q$, we can easily construct an epimorphism $G\to \Z(p^\infty)$, whose kernel $K$  is fully invariant in $G$ (in fact, as $G$ is endorigid, every subgroup of $G$ is fully invariant in $G$). Finally, notice that $s(G)=0$ by Theorem \ref{endorigid} and $s_0(\Z(p^\infty))=\infty$ by Example \ref{prufer}.
\item[(b)] Consider the group of Example \ref{mono-sub}(b), that is an abelian $p$-group $G$ with a basic subgroup $B$ such that $ns(G)=0$ and $ns(B)=\infty$. By a celebrated theorem of Szele (see for example \cite[36.1]{F}), $B$ is a quotient of $G$.
\end{itemize}
\end{example}

Lemma \ref{LemmaY} implies that $\mathfrak S$, $\Sns$ and $\S0$ are closed under taking direct summands. We consider now the stability of these classes for direct sums. As observed after Lemma \ref{LemmaY}, $\mathfrak S$, $\Sns$ and $\S0$ are not closed under taking finite direct sums (and so they are not closed also under extensions). On the other hand, we saw in Lemma \ref{LemmaY-fi} that, given an abelian group $G$ that is the direct sum of a family of fully invariant subgroups belonging to $\mathfrak S$ (respectively, $\Sns$, $\S0$), then $G$ belongs to $\mathfrak S$ (respectively, $\Sns$, $\S0$) as well. 
In Theorem \ref{dec} we give a partial converse to this fact when $G$ is torsion-free. Namely we show that, in order to have $G\in \mathfrak{S}$ (respectively, $\Sns$),  the hypothesis of being fully invariant is necessary for the summands of $G$.

\begin{theorem}\label{dec}
Let $A$ be a set, $H_\alpha$ a torsion-free abelian group for every $\alpha\in A$, and $G=\bigoplus_{\alpha\in A} H_\alpha$. Then
\begin{itemize}
\item[(a)] $s(G)=0$ if and only if $H_\alpha$ is fully invariant in $G$ for every $\alpha$ in $A$ and $s(H_\alpha)=0$ for every $\alpha\in A$;
\item[(b)] $ns(G)=0$ if and only if $H_\alpha$ is fully invariant in $G$ for every $\alpha$ in $A$ and $ns(H_\alpha)=0$ for every $\alpha\in A$.
\end{itemize}
\end{theorem}
\begin{proof}
(a) If $H_\alpha$ is fully invariant in $G$ and $s(H_\alpha)=0$ for every $\alpha\in A$, then $s(G)=0$ by Lemma \ref{LemmaY-fi}.
It remains to prove the converse implication. If $s(G)=0$, then $s(H_\alpha)=0$ for all $\alpha\in A$ by Lemma \ref{LemmaY}(a). Assume that $H_\alpha$ is not fully invariant for some $\alpha\in A$. Then there exist $\beta\neq\alpha$ in $A$ and a non-zero homomorphism $\psi: H_\alpha \to H_\beta$. Since $\psi$ is non-zero, there exists $x\in H_\alpha\setminus\{0\}$ such that $\psi(x)\neq0$. Then $S=\{(-n \psi(x),x)\}_{n\in\N}$ is a pseudostring of the endomorphism $\Psi$ of $H_\alpha\oplus H_\beta$ given by the matrix $\begin{pmatrix}1 & \psi \\ 0 & 1\end{pmatrix}$. Since $H_\alpha\oplus H_\beta$ is torsion-free, it follows that $S$ is a (non-singular) string of $\Psi$. Consequently, $s(H_\alpha\oplus H_\beta)=\infty$ and hence $s(G)=\infty$ by Lemma \ref{LemmaY}(a). 

\medskip
(b) Proceed analogously to (a).
\end{proof}

We now consider the class $\S0$.
It presents a different behaviour to that of $\mathfrak S$ and $\Sns$ described by Theorem \ref{dec}; indeed, the implication given by Lemma \ref{LemmaY-fi} holds also for $s_0(-)$, while the abelian group $G=\Z\oplus\Z$ belongs to $\mathfrak S_0$, even if $\Z\oplus0$ and $0\oplus\Z$ are not fully invariant subgroups of $G$.

Moreover, the following inequality holds:

\begin{lemma}\label{WADT}
Let $G$ be an abelian group and $H$ a fully invariant subgroup of $G$. Then 
$$s_0(G)\leq s_0(H)+s_0(G/H).$$
Equality holds if $H$ is a direct summand of $G$.
\end{lemma}
\begin{proof}
We have to prove that, if $s_0(G)=\infty$, then at least one between $s_0(H)$ and $s_0(G/H)$ is infinite. Suppose that there exists $\phi\in\End(G)$ such that $s_0(\phi)=\infty$. 
By Proposition \ref{AAA}(c), $\ker\phi\restriction_{\sc\phi}\neq0$. Let $K=\ker\phi\restriction_{\sc\phi}=\ker\phi\cap\sc\phi$.
 If $K\subseteq H$, since $\ker\phi\cap H=\ker\phi\restriction_H$ and $\sc\phi\cap H=\sc\phi\restriction_H$ and $\ker\phi\restriction_H\cap\sc\phi\restriction_H\subseteq K$, it follows that $K=\ker\phi\restriction_H\cap\sc\phi\restriction_H$. By Proposition \ref{AAA}(c) $s_0(\phi\restriction_H)=\infty$ and so $s_0(H)=\infty$.
 If $K\not\subseteq H$, consider the endomorphism $\overline\phi:G/H\to G/H$ induced by $\phi$ and the canonical projection $\pi:G\to G/H$. Since $\pi(\ker\phi)\subseteq\ker\overline\phi$ and $\pi(\sc\phi)\subseteq\sc {\overline{\phi}}$, $\pi(K)\subseteq \ker\overline\phi\cap\sc{\overline{\phi}}$. As $K\not\subseteq H$, $\pi(K)\neq0$, so $\ker\overline\phi\cap\sc{\overline{\phi}}\neq 0$ and hence $s_0(\overline\phi)=\infty$ by Proposition \ref{AAA}(c). Consequently, $s_0(G/H)=\infty$.

The last statement follows from the inequality and from Lemma \ref{LemmaY}(c).
\end{proof}

The inequality in Lemma \ref{WADT} can be strict as the following example shows.

\begin{example}
%
Consider the abelian group $G$ constructed in Example \ref{mono-sub}(a), which has $s_0(G)=0$. It is shown in Example \ref{mono-subbis}(a) that $s_0(G/pG)=\infty$, as $G$ has infinite torsion-free rank. Since $pG$ is a fully invariant subgroup of $G$, this provides an example of how the inequality of Lemma \ref{WADT} can be strict.
%
\end{example}

\smallskip
Even if we have seen that $\S0$ is not closed under taking extensions (as $\mathfrak S$ and $\Sns$), we have the following consequence of Lemma \ref{WADT}.

\begin{corollary}
Let $G$ be an abelian group and $H$ a fully invariant subgroup of $G$. If $s_0(H)=0$ and $s_0(G/H)=0$, then $s_0(G)=0$. 
\end{corollary}

At this stage it should be clear that none of the classes $\mathfrak S$, $\Sns$ and $\S0$ is closed under taking direct products. We want now to study the case of an abelian group $G$ that is the direct product of fully invariant subgroups. We will see that (as for the direct sums) the behavior of $\S0$ is quite different from that of $\mathfrak S$ and $\Sns$. Let us start with the following

\begin{example}\label{ex-dir-prod}
Let $G=\prod_p \Z(p)$. Each $\Z(p)$ is fully invariant in $G$ and $s(\Z(p))=0$ as $\Z(p)$ is finite. Nevertheless, $s(G)=ns(G)=\infty$. 

Indeed, for every prime $p$, let $z_p	\in\Z$ be, modulo $p$, a generator of the cyclic multiplicative group $\Z(p)^*$ of all non-trivial elements of $\Z(p)$. Consider the endomorphism
$$\phi_p:\Z(p)\to \Z(p), \ \ \ \phi_p(x)=z_p\cdot x;$$ 
it is easily seen that $\phi_p$ is an automorphism of order $p-1$. Now consider on $G$ the diagonal endomorphism $\phi=(\phi_p)_p$ of the $\phi_p$. Then $\phi$ is an automorphism of $G$, such that $Q\Per(\phi)\neq G$, since the orbit of the element $(1_{\Z(p)})_p$ under the action of $\phi$ is infinite. By Proposition \ref{AAA} we conclude that $ns(\phi)=\infty$, and so $ns(G)=\infty$.
\end{example}

As announced above, we show now that we cannot find an analog of Example \ref{ex-dir-prod} for the class $\S0$.

\begin{lemma}
Let $G$ be an abelian group such that $G=\prod_{\alpha\in A}H_\alpha$, where $A$ is a set and $H_\alpha$ is a fully invariant subgroup of $G$ for every $\alpha\in A$. Then $s_0(G)=0$ if and only if $s_0(H_\alpha)=0$ for every $\alpha\in A$.
\end{lemma}
\begin{proof}
If $s_0(G)=0$, then $s(H_\alpha)=0$ for every $\alpha\in A$ by Lemma \ref{LemmaY}(c). To prove the converse implication, suppose that $s_0(H_\alpha)=0$ for every $\alpha\in A$ and consider  $\phi\in\End(G)$. We have to verify that $s_0(\phi)=0$. Denote by $\phi_\alpha$ the restriction of $\phi$ to $H_\alpha$, with $\alpha\in A$. Since $\phi$ acts componentwise, each $\phi_\alpha$ is an endomorphism of $H_\alpha$. By Lemma \ref{sc-sp}, we have that $\sc \phi=\prod_{\alpha\in A}\sc \phi_\alpha$. Since, by hypothesis, $s_0(\phi_\alpha)=0$, we obtain that $\ker(\phi_\alpha)\cap \sc \phi_\alpha=0$ for all $\alpha\in A$ by Proposition \ref{AAA}(c). It follows that $\sc \phi\cap \ker(\phi)=\prod_{\alpha\in A}(\sc\phi_\alpha\cap\ker\phi_\alpha)=0$, and so $s_0(\phi)=0$ again by Proposition \ref{AAA}(c).
\end{proof}

\section{Strings and Hopficity}\label{SHopf}

In this section we study the relations between the three string numbers and the Hopfian property. In \cite{DGSZ} it is noted that a consequence of \cite[Proposition 2.9]{DGSZ} is that an abelian group $G$ with zero algebraic entropy is necessarily co-Hopfian (i.e., every monomorphism $G\to G$ is an automorphism of $G$). It is  also asked about a connection between zero algebraic entropy and the Hopfian property. 
Moreover, in \cite{GG} it is proved that the reduced torsion-free abelian groups with zero adjoint algebraic entropy are Hopfian.

\smallskip
First of all, Theorem \ref{Hopf} shows that Hopficity is a necessary property for an abelian group to have null string number zero. In other words, $\S0$ is contained in the class of all Hopfian abelian groups. 

\begin{theorem}\label{Hopf}
Every abelian group $G$ with $s_0(G)=0$ is Hopfian.  
\end{theorem}
\begin{proof}
Let $\phi\in\End(G)$ and suppose that $\phi$ is surjective. If $\phi$ is non-injective, then $s_0(\f)=\infty$ by Proposition \ref{AAA}(c). This shows that $\phi$ has to be injective and so $G$ is Hopfian.
\end{proof}

Since $\mathfrak S\subseteq \S0$, we have the following immediate consequence of Theorem \ref{Hopf}. 

\begin{corollary}
Every abelian group $G$ with $s(G)=0$ is Hopfian.  
\end{corollary}

It is not true that an abelian group $G$ in $\Sns$ is necessarily Hopfian. Indeed, $\Z(p^\infty)$ is a torsion abelian group with $ns(\Z(p^\infty))=0$ and it is non-Hopfian. On the other hand $\Z(p^\infty)$ is a divisible abelian $p$-group, so the following problem is open:

\begin{problem}
Find a reduced abelian $p$-group $G$ with $ns(G)=0$ and such that $G$ is non-Hopfian.
\end{problem} 

Also for torsion-free abelian groups the same problem is open:

\begin{problem}\label{tfnH-ns}
Find a torsion-free abelian group $G$ with $ns(G)=0$ and such that $G$ is non-Hopfian.
\end{problem} 

All our examples of torsion-free abelian groups in $\Sns$ either have finite torsion-free rank or are endorigid. In both cases the abelian groups in question are Hopfian. Furthermore, the non-Hopfian torsion-free abelian groups $G$ that we consider are either infinite direct sums of some torsion-free abelian group or have $p^{\omega}G\neq 0$ for some prime $p$. In both cases $ns(G)=\infty$.

\medskip
The next example shows that the converse implication of Theorem \ref{Hopf} does not hold true, namely Hopficity is a necessary but not sufficient condition for an abelian group to have null string number zero. In particular, this answers negatively \cite[Question 3.13]{DGV}. Moreover, this example shows that the converse implication of Corollary \ref{B-2} does not hold true as well.

\begin{example}\label{hopf-s0=infty}
For a prime $p$, there exists an abelian $p$-group $G$ such that:
\begin{itemize}
\item[(a)] $G$ is Hopfian;
\item[(b)] $|G|=\cont$ (this implies that $r_p(G)$ is infinite);
\item[(c)] $s_0(G)=\infty$;
\item[(d)] $ns(G)=0$.
\end{itemize}
The construction of an abelian $p$-group $G$ satisfying (a) and (b) 
 is given in \cite[Theorem 16.4]{P}. To see that $G$ satisfies (c), since $r_p(G)$ is infinite by (b),
it is enough to apply Theorem \ref{ThB}. Moreover, as observed after \cite[Corollary 5.3]{DGSZ}, $\ent(G)=0$ and so $ns(G)=0$ as well in view of \eqref{ent=0->ns=0}.
\end{example}

The following  problem remains open.

\begin{problem}\label{Ht-ns}
Find a Hopfian torsion abelian group $G$ such that $ns(G)=\infty$.
\end{problem}

The torsion-free version of this problem is easily solved by $\Z^2$, which is an Hopfian abelian group not in $\Sns$ by Example \ref{Z}(b). It is less easy to find examples of Hopfian torsion-free abelian groups not in $\S0$. We now provide such an example, which is a modification of a group constructed in \cite[Example 2]{C}.

\begin{example}\label{tf-hopf}
We give the construction of a torsion-free Hopfian abelian group $G$ with $s_0(G)=\infty$.

Let $V$ be a $\Q$-vector space whose base is $\{a_k,b_k\}_{k\in\Z}$ and let $p,q_k$ $(k\in\Z)$ be distinct primes. Define
$$G=\bigoplus_{k\in\Z} \left\langle p^{-\infty}a_k,q_k^{-\infty}b_k, \frac{1}{p}(a_k+b_k)\right\rangle,$$
where $p^{-\infty}x$ stands for the set $\{p^{-n}x\}_{n\in\N}$. By the proof of \cite[Example 2]{C}, $G$ is Hopfian.
To show that $s_0(G)=\infty$, consider the assignments
$$a_k\mapsto a_{k+1} \text{ (if $k< 0$)},\ \ a_k\mapsto 0 \text{ (if $k\geq 0$)},\ \ b_k\mapsto 0 \text{ (for every $k\in\Z$);}$$
it is easy to prove that they induce a group homomorphism $\psi:G\to G$. 
Now it is clear that $\{a_k\}_{k<0}$ is a null string of $\psi$.
\end{example}

The next example (due to Corner) shows that there exists an abelian group $G$ such that $s_0(G)=0$ and $s_0(G\oplus G)=\infty$; as we already mentioned, this proves that $\S0$ is not closed under taking finite direct sums and extensions.

\begin{example}\label{hopf+hopf-non-hopf}
The group $G$ constructed in \cite[Example 3]{C} has the property that any non-zero endomorphism of $G$ is injective. In particular $G$ is Hopfian and $s_0(G)=0$ by Proposition \ref{AAA}(c). Furthermore $G\oplus G$ is not Hopfian; hence $s_0(G\oplus G)=\infty$ by Theorem \ref{Hopf}.
\end{example}

\section{Hereditary string numbers}\label{her-sec}

In this section we consider the problem of characterizing the classes $\widetilde{\mathfrak S}$, $\widetilde{\mathfrak S}_{ns}$ and $\widetilde{\mathfrak S}_0$. Theorem \ref{main-tilde} will completely solve it. We consider separately the torsion and the torsion-free case, which will lead to the proof of the general theorem.

\medskip
In view of \eqref{s=ns+s0}, 
\begin{equation}\label{s=ns+s0-tilde}
\widetilde s(G)=\widetilde{ns}(G)+\widetilde{s_0}(G);
\end{equation}
this is equivalent to $\widetilde{\mathfrak S}=\widetilde{\mathfrak S}_{ns}\cap \widetilde{\mathfrak S}_0$. Furthermore, by definition, each of three classes $\widetilde{\mathfrak S},$ $\widetilde{\mathfrak S}_{ns}$ and  $\widetilde{\mathfrak S}_0$ is closed under taking subgroups.

\begin{lemma}\label{s=0->r-fin}
Let $G$ be an abelian group. 
\begin{itemize}
\item[(a)] If either $\widetilde s(G)=0$, or $\widetilde{ns}(G)=0$, or $\widetilde{s_0}(G)=0$, then $r_p(G)$ is finite for every prime $p$ and $r_0(G)$ is finite.
\item[(b)] If either $\widetilde s(G)=0$, or $\widetilde{ns}(G)=0$, then $r_0(G)\leq1$.
\end{itemize}
\end{lemma}
\begin{proof}
(a) If $r_p(G)$ is infinite for some prime $p$, then $G$ contains a subgroup isomorphic to $\Z(p)^{(\N)}$. Since $s(\Z(p)^{(\N)})=ns(\Z(p)^{(\N)})=s_0(\Z(p)^{(\N)})=\infty$ by Example \ref{Z}(d), it follows that $\widetilde s(G)=\widetilde{ns}(G)=\widetilde{s_0}(G)=\infty$.
If $r_0(G)$ is infinite, then $G$ contains a subgroup isomorphic to $\Z^{(\N)}$. Since $s(\Z^{(\N)})=ns(\Z^{(\N)})=s_0(\Z^{(\N)})=\infty$ by Example \ref{Z}(d), it follows that $\widetilde s(G)=\widetilde{ns}(G)=\widetilde{s_0}(G)=\infty$.

\smallskip
(b) If $r_0(G)>1$, then $G$ contains a subgroup isomorphic to $\Z^2$. Since $s(\Z^2)=ns(\Z^2)=\infty$ by Example \ref{Z}(b), it follows that $\widetilde s(G)=\widetilde{ns}(G)=\infty$.
\end{proof}

The next lemma describes completely the torsion abelian groups in $\widetilde{\mathfrak S}=\widetilde{\mathfrak S}_0$. In particular, it adds equivalent conditions to Theorem \ref{ThB}.

\begin{lemma}\label{ss=0}
Let $G$ be a torsion abelian group. Then the following conditions are equivalent:
\begin{itemize}
\item[(a)] $\widetilde s(G)=0$ ($\Leftrightarrow$ $\widetilde{s_0}(G)=0$);
\item[(b)] $s(G)=0$ ($\Leftrightarrow$ $s_0(G)=0$);
\item[(c)] $t_p(G)$ is finite for every prime $p$.
\end{itemize}
In particular, $\widetilde s(G)=\widetilde{s_0}(G)$.
\end{lemma}
\begin{proof}
The equivalences follow directly from Theorem \ref{ThB}, noting that the property in (c) is hereditary.
\end{proof}

The condition of Corollary \ref{B-2}, which is  sufficient but not necessary in order to have $ns(G)=0$ for a torsion abelian group $G$, turns out to be equivalent to $\widetilde{ns}(G)=0$:

\begin{lemma}\label{ns=0}
Let $G$ be a torsion abelian group. Then $\widetilde{ns}(G)=0$ if and only if $r_p(G)$ is finite for every prime $p$.
\end{lemma}
\begin{proof}
If $\widetilde{ns}(G)=0$, then $r_p(G)$ is finite for every prime $p$ by Lemma \ref{s=0->r-fin}(a).
Suppose that $r_p(G)$ is finite for every prime $p$. Then $r_p(H)$ is finite for every prime $p$ and every subgroup $H$ of $G$. By Corollary \ref{B-2}, $ns(H)=0$ for every $H$, and hence $\widetilde{ns}(G)=0$.
\end{proof}

For a torsion abelian group $G$, in view of Lemma \ref{ss=0}, the condition $\widetilde{s_0}(G)$ becomes equivalent to $\widetilde s(G)=0$. So in this case $\widetilde{ns}(G)=0$ is weaker than $\widetilde{s_0}(G)=0$, as its value on $\Z(p^\infty)$ shows; indeed, $\Z(p^\infty)$ is an infinite abelian $p$-group of finite $p$-rank. In other words $\widetilde{\mathfrak S}_{ns}\cap\mathfrak T\supsetneq\widetilde{\mathfrak S}\cap \mathfrak T= \widetilde{\mathfrak S}_0\cap\mathfrak T$.

\medskip
Lemmas \ref{ss=0} and \ref{ns=0} characterize completely the torsion abelian groups respectively in $\widetilde{\mathfrak S}=\widetilde{\mathfrak S}_0$ and in $\widetilde{\mathfrak S}_{ns}$. We now pass to consider the torsion-free case.

\begin{lemma}\label{tf-rk<inf}\label{last'}\label{last}
Let $G$ be a torsion-free abelian group. Then:
\begin{itemize}
\item[(a)] $\widetilde{s_0}(G)=0$ if and only if $r_0(G)$ is finite;
\item[(b)] $\widetilde{s}(G)=\widetilde{ns}(G)$;
\item[(c)] $\widetilde{s}(G)=0$ if and only if $r_0(G)\leq1$ and no infinity appears in the type of $G$. 
\end{itemize}
\end{lemma}
\begin{proof}
(a) If $\widetilde{s_0}(G)=0$, then $r_0(G)$ is finite by Lemma \ref{s=0->r-fin}(a).
If $r_0(G)$ is finite, then $r_0(H)$ is finite for every subgroup $H$ of $G$. So Lemma \ref{s0(fin-rank)=0} gives $s_0(H)=0$ for every $H$. This proves that $\widetilde{s_0}(G)=0$.

\smallskip\noindent
(b) If $r_0(G)$ is infinite, then $G$ contains a subgroup isomorphic to $\Z^{(\N)}$, which has $ns(\Z^{(\N)})=s(\Z^{(\N)})= \infty$ by Example \ref{Z}(d), so $\widetilde s(G)=\widetilde{ns}(G)=\infty$ as well. If $r_0(G)$ is finite, then $\widetilde{s_0}(G)=0$ thanks to part (a). By \eqref{s=ns+s0-tilde} we have $\widetilde{s}(G)=\widetilde{ns}(G)$.

\smallskip\noindent
(c) If $r_0(G)>1$, then $G$ contains an isomorphic copy of $\Z^2$ and $s(\Z^2)=\infty$ by Example \ref{Z}(b); hence $\widetilde s(G)=\infty$. Now Corollary \ref{tf:rk1} concludes the proof, since if $A, B$ are of rank $1$, then $A\leq B$ if and only if $t(A)\leq t(B)$.
\end{proof}




Let $\mathfrak F$ denote the class of all the torsion-free abelian groups. In view of Lemma \ref{last'}, for $G$ in $\mathfrak F$ the condition $\widetilde{ns}(G)=0$ is equivalent to $\widetilde s(G)=0$. So in this case $\widetilde{s_0}(G)=0$ is weaker than $\widetilde{ns}(G)=0$, as its value on $\Z^2$ shows; indeed, $\Z^2$ has finite torsion-free rank but strictly grater that $1$. In other words $\widetilde{\mathfrak S}_0\cap\mathfrak F\supsetneq\widetilde{\mathfrak S}\cap \mathfrak F= \widetilde{\mathfrak S}_{ns}\cap\mathfrak F$.

\medskip
Lemmas \ref{ss=0} and \ref{ns=0} characterize completely the classes $\widetilde{\mathfrak S}\cap \mathfrak T$, $\widetilde{\mathfrak S}_{ns} \cap\mathfrak T$ and $\widetilde{\mathfrak S}_0\cap \mathfrak T$, while Lemma \ref{tf-rk<inf} characterizes completely the classes $\widetilde{\mathfrak S}\cap \mathfrak F$, $\widetilde{\mathfrak S}_{ns} \cap\mathfrak F$ and $\widetilde{\mathfrak S}_0\cap \mathfrak F$. Since every abelian group is the extension of a torsion and a torsion-free abelian group, the last step for the complete characterization of $\widetilde{\mathfrak S}$, $\widetilde{\mathfrak S}_{ns} $ and $\widetilde{\mathfrak S}_0$  is to study whether these classes are closed under taking extensions.

In this direction we start proving the following inequalities for the hereditary string numbers with respect to the torsion part. Note that Corollary \ref{equ} of Theorem \ref{main-tilde} will show that actually equality holds (and that Lemma \ref{WADT'} is applied in the proof of Theorem \ref{main-tilde}).

\begin{lemma}\label{WADT-tilde}\label{WADT'}\label{in3}
Let $G$ be an abelian group. Then:
\begin{itemize}
\item[(a)] $\widetilde{s}(G)\leq \widetilde{s}(t(G))+\widetilde{s}(G/t(G));$
\item[(b)] $\widetilde{ns}(G)\leq \widetilde{ns}(t(G))+\widetilde{ns}(G/t(G));$
\item[(c)] $\widetilde{s_0}(G)\leq\widetilde{s_0}(t(G))+\widetilde{s_0}(G/t(G)).$
\end{itemize}
\end{lemma}
\begin{proof}
(a) It will follow by (b)-(c) and the equality $\widetilde{s}(G)=\widetilde{ns}(G)+\widetilde{s_0}(G)$.

\smallskip
(b) It suffices to prove that $\widetilde{ns}(G)=0$, if $\widetilde{ns}(t(G))=0$ and $\widetilde{ns}(G/t(G))=0$. So assume that $\widetilde{ns}(t(G))=0$ and $\widetilde{ns}(G/t(G))=0$. In view of Lemma \ref{ns=0}, $\widetilde{ns}(t(G))=0$ means that all the $p$-ranks of $t(G)$ are finite.
Let $H$ be an arbitrary subgroup of $G$; we have to prove that $ns(H)=0$. Since $t(H)\subseteq t(G)$, it follows that $r_p(H)$ is finite for every prime $p$. Moreover, $ns(H/t(H))=0$, as $H/t(H)\cong (H+t(G))/t(G)\leq G/t(G)$. 

Let $\phi\in\End(H)$. By Lemma \ref{s-c} we can suppose without loss of generality that $\phi$ is surjective, that is, $H=\sc\phi$. Assume that $S=\{x_n\}_{n\in\N}$ is a string of $\phi$ in $H$. 
Since $ns(H/t(H))=0$, in particular $ns(\overline\phi)=0$, where $\overline{\phi}:H/t(H)\to H/t(H)$ is the endomorphism induced by $\phi$ and $\overline\phi$ is surjective as $\phi$ is surjective. By Proposition \ref{AAA}(b)  the element $x_0+t(H)$ is a quasi-periodic point of $\overline\phi$. This means that there exist $j\in\N$ and $k\in\N_+$  such that $\phi^{j+k}(x_0)-\phi^{j}(x_0)=t\in t(H)$. We can suppose without loss of generality that $j=0$. Consider now the trajectory $T=T(\phi,\langle t\rangle)$ and let $m\in\N$ be the order of $t$; then $T\subseteq t(H)[m]$. Since all the $p$-ranks of $t(H)$ are finite, the fully invariant subgroup $t(H)[m]$ is finite; therefore $T$ is finite as well. The set $\{\phi^{sk}(x_0):s\in\N\}$ is contained in $x_0+T$, which is finite. Consequently, $x_0$ is a quasi-periodic point of $\phi$ and so $S$ is singular by Lemma \ref{sing}. Therefore, we have proved that every string of $\phi$ is singular. Consequently, $ns(\phi)=0$. Since $\phi$ was chosen arbitrarily, this argument shows that $ns(H)=0$. Hence, as $H$ is an arbitrary subgroup of $G$, we can conclude that $\widetilde{ns}(G)=0$.

\smallskip
(c) It suffices to show that, if $\widetilde{s_0}(G)=\infty$, then at least one between $\widetilde{s_0}(t(G))$ and $\widetilde{s_0}(G/t(G))$ is infinite.
So assume that $\widetilde{s_0}(G)=\infty$. Then there exists a subgroup $K$ of $G$ such that $s_0(K)=\infty$. Since $K\cap t(G)=t(K)$ is fully invariant in $K$, Lemma \ref{WADT} aplies to give $s_0(K)\leq s_0(K\cap t(G)) + s_0(K/K\cap t(G))$. If $s_0(K)=\infty$, then $\widetilde{s_0}(K)=\infty$. If $s_0(K/K\cap t(G))=\infty$, since $K/K\cap t(G)\cong K+t(G)/t(G)$, $s_0(K+t(G)/t(G))=\infty$ as well, and so $\widetilde{s_0}(G/t(G))=\infty$.
\end{proof}

We are now ready to prove the main theorem of this section:

\begin{theorem}\label{main-tilde}
Let $G$ be an abelian group. Then:
\begin{itemize}
\item[(a)] $\widetilde{s}(G)=0$ if and only if $t_p(G)$ is finite for every prime $p$ and $G/t(G)$ is a torsion-free group of torsion-free rank $1$ such that no infinity appears in its type; 
\item[(b)] $\widetilde{ns}(G)=0$ if and only if $r_p(G)$ is finite for every prime $p$ and $G/t(G)$ is a torsion-free group of torsion-free rank $1$ such that no infinity appears in its type;
\item[(c)] $\widetilde{s_0}(G)=0$ if and only if $t_p(G)$ is finite for every prime $p$ and $r_0(G)$ is finite.
\end{itemize}
\end{theorem}
\begin{proof}
(a) Follows from (b) and (c) using (\ref{s=ns+s0-tilde}).

\smallskip
(b) Let $\widetilde{ns}(G)=0$; then Lemma \ref{s=0->r-fin} implies that $r_p(G)$ is finite for every prime $p$ and $r_0(G/t(G))\leq 1$. We want to show that $p^{\omega}(G/t(G))=0$ for every prime $p$. So fix an arbitrary prime $p$ and suppose, looking for a contradiction, that $p^{\omega}(G/t(G))\neq 0$. 

Since the $p$-rank of $G$ is finite, then $G=t_p(G)\oplus C$ for some subgroup $C$ of $G$. Furthermore, $C/t(C)\cong G/t(G)$ and $t_p(C)=0$. Consider the endomorphism
$$\mu_p:C\to C, \ \ \ x\mapsto px.$$
We find a non-singular string of $\mu_p$.

Consider the endomorphism induced by $\mu_p$ on $C/t(C)$, that is, 
$$\overline{\mu_p}:C/t(C)\to C/t(C),\ \ \ x+t(C)\mapsto px+t(C).$$
Since $p^{\omega}(C/t(C))\neq 0$, $\overline{\mu_p}$ admits a non-singular string $\overline{S}=\{x_n+t(C)\}_{n\in\N}$ by \cite[Corollary 3.21(b)]{DGV}.  This implies $p^nx_{n}-x_0\in t(C)$ for every $n\in\N.$
Let $$t_n=p^nx_{n}-x_0\in t(C),\ \ \ \text{for every}\ n\in\N.$$
By construction $t_p(C)=0$, so $t(C)$ is $p$-divisible and then there exist $s_1,\dots,s_n,\ldots\in t(C)$ such that $p^ns_n=t_n$ for all $n\in\N$. 
Set $y_n=x_n-s_n$ for every $n\in\N$; these are torsion-free elements such that $p^ny_n=x_0$ for every $n\in\N$.

We verify that $S=\{y_n\}_{n\in\N}$ is a non-singular string of $\mu_p$. 
First we see that $S$ is a pseudostring of $\mu_p$. Indeed, by the definition of the $y_n$, we have
$$p^n(py_{n+1}-y_n)=p^{n+1}y_{n+1}-p^ny_n=y_0-y_0=0,$$ 
and so $py_{n+1}-y_n\in t_p(C)=0$; therefore, $py_{n+1}=y_n$ for every $n\in\N$, which proves that $S$ is a pseudostring of $\mu_p$.

That $S$ is a non-singular string follows from the fact that $\overline{S}$ is a non-singular string, since $x_n+t(C)=y_n+t(C)$ for every $n\in\N$ (i.e., $\overline S=\{y_n+t(C)\}_{n\in\N}$).
This implies that $\widetilde{ns}(C)=\infty$, and so $\widetilde{ns}(G)=\infty$, that is the contradiction we are looking for.

To prove the converse implication, suppose that $r_p(G)$ is finite for every prime $p$ and that $G/t(G)$ has torsion-free rank $1$ and no infinity appears in its type. Then $\widetilde{ns}(t(G))=0$ and $\widetilde{ns}(G/t(G))=0$ respectively by Lemma \ref{ns=0} and by Lemma \ref{last}. Now we can conclude by Lemma \ref{WADT'}(b) that $\widetilde{ns}(G)=0$.

\smallskip
(c) If $\widetilde{s_0}(G)=0$, then $t_p(G)$ is finite for every prime $p$ by Lemma \ref{ss=0}, and $r_0(G/t(G))$ is finite by Lemma \ref{s=0->r-fin}(a). 

To prove the converse implication, suppose that $t_p(G)$ is finite for every prime $p$ and that $r_0(G/t(G))$ is finite. Then $\widetilde{s_0}(t(G))=0$ and $\widetilde{s_0}(G/t(G))=0$ respectively by Lemma \ref{ss=0} and by Lemma \ref{last'}. Now we can conclude that $\widetilde{s_0}(G)=0$ by Lemma \ref{WADT-tilde}(c).
\end{proof}

As a first corollary of this theorem, we have that the inequalities in Lemma \ref{WADT'} are equalities:

\begin{corollary}\label{equ}
Let $G$ be an abelian group. Then:
\begin{itemize}
\item[(a)]$\widetilde{s}(G)=\widetilde{s}(t(G))+\widetilde{s}(G/t(G))$;
\item[(b)]$\widetilde{ns}(G)=\widetilde{ns}(t(G))+\widetilde{ns}(G/t(G))$;
\item[(c)]$\widetilde{s_0}(G)=\widetilde{s_0}(t(G))+\widetilde{s_0}(G/t(G))$.
\end{itemize}
\end{corollary}



In the next corollary we show that hereditarily Hopfian abelian groups are precisely the groups in $\widetilde{\mathfrak S}_0$. So, from Theorem \ref{main-tilde}(c), we get the following structure theorem for hereditarily Hopfian abelian groups.

\begin{corollary}\label{her-H}
Let $G$ be an abelian group. Then the following conditions are equivalent:
\begin{itemize}
\item[(a)]$\widetilde{s_0}(G)=0$;
\item[(b)] $t_p(G)$ is finite for every prime $p$ and $r_0(G)$ is finite;
\item[(c)]$G$ is hereditarily Hopfian.
\end{itemize}
\end{corollary}
\begin{proof}
(a)$\Leftrightarrow$(b) is Theorem \ref{main-tilde}(c).

\smallskip
(a)$\Rightarrow$(c) Suppose $\widetilde{s_0}(G)=0$. So $s_0(H)=0$ for every subgroup $H$ of $G$. Therefore $G$ is hereditarily Hopfian by Theorem \ref{Hopf}.

\smallskip
(c)$\Rightarrow$(a) On the other hand, if $G$ is hereditarily Hopfian, let $H$ be a subgroup of $G$ and $\f\in\End(H)$. Since $\f\restriction_{\sc\f}:\sc\f\to\sc\f$ is surjective, and $\sc\f$ is Hopfian, $\f\restriction_{\sc\f}$ is injective. By Proposition \ref{AAA}(c), $s_0(\f)=0$. This proves that $s_0(H)=0$ and so that $\widetilde{s_0}(G)=0$.
\end{proof}

\medskip
We end this section adding to Table \ref{table1} the values of the hereditarily string numbers of the examples considered there:

\begin{center}
\begin{tabular}{|c|ccc|ccc|}
\hline
 & $s(-)$ & $ns(-)$ & $s_0(-)$ & $\widetilde s(-)$ & $\widetilde{ns}(-)$ & $\widetilde{s_0}(-)$ \\
 \hline
$\Z$ & $0$ & $0$ & $0$ & $0$ & $0$ & $0$ \\
$\Z^2$ & $\infty$ & $\infty$ & $0$ & $\infty$ & $\infty$ & $0$ \\
$\Q$ & $\infty$ & $\infty$ & $0$ & $\infty$ & $\infty$ & $0$ \\
$\mathbb J_p$ & $\infty$ & $\infty$ & $0$ & $\infty$ & $\infty$ & $\infty$ \\
$\Z(p^\infty)$ & $\infty$ & $0$ & $\infty$ & $\infty$  & $0$ & $\infty$ \\
$\Q/\Z$ & $\infty$ & $0$ & $\infty$ & $\infty$  & $0$ & $\infty$  \\
$B_p$ & $\infty$ & $\infty$ & $\infty$ & $\infty$ & $\infty$ & $\infty$ \\
$K^{(\N)}$ & $\infty$ & $\infty$ & $\infty$ & $\infty$ & $\infty$  & $\infty$ \\
\hline
\end{tabular}
\captionof{table}{String numbers and hereditary string numbers}\label{table2}
\end{center}

\end{document}